\documentclass[reqno]{amsart}
\usepackage{amssymb}
\usepackage{amsmath}
\usepackage{hyperref}
\usepackage{a4wide}
\hypersetup{ colorlinks = true, urlcolor = blue, linkcolor = blue, citecolor = red }
\usepackage{xcolor}
\usepackage{enumitem}
\usepackage{esint}
\makeatletter
\def\namedlabel#1#2{\begingroup
	#2%
	\def\@currentlabel{#2}%
	\phantomsection\label{#1}\endgroup
}
\usepackage{varwidth}
\usepackage{tasks}
\DeclareMathOperator{\dv}{div}

\newcommand{\RR}{\mathbb{R}}
\newcommand{\mA}{\mathcal{A}}

\newcommand{\na}{\nabla}
\newcommand{\pa}{\partial}
\newcommand{\La}{\Lambda}
\newcommand{\la}{\lambda}
\newcommand{\ep}{\epsilon}
\newcommand{\data}{\mathit{data}}
\usepackage[T1]{fontenc}
\usepackage[utf8]{inputenc}
\usepackage[french,english]{babel}

\makeatletter
\newenvironment{abstracts}{%
  \ifx\maketitle\relax
    \ClassWarning{\@classname}{Abstract should precede
      \protect\maketitle\space in AMS document classes; reported}%
  \fi
  \global\setbox\abstractbox=\vtop \bgroup
    \normalfont\Small
    \list{}{\labelwidth\z@
      \leftmargin3pc \rightmargin\leftmargin
      \listparindent\normalparindent \itemindent\z@
      \parsep\z@ \@plus\p@
      
      \itemsep\medskipamount
    }%
}{%
  \endlist\egroup
  \ifx\@setabstract\relax \@setabstracta \fi
}

\newcommand{\abstractin}[1]{%
  \otherlanguage{#1}%
  \item[\hskip\labelsep\scshape\abstractname.]%
}
\makeatother

\theoremstyle{plain}
\newtheorem{theorem}{Theorem}[section]
\newtheorem{lemma}[theorem]{Lemma}
\newtheorem{corollary}[theorem]{Corollary}

\newtheorem{definition}[theorem]{Definition}
\newtheorem{proposition}[theorem]{Proposition}

\newtheorem{remark}[theorem]{Remark}
\def\Xint#1{\mathchoice
	{\XXint\displaystyle\textstyle{#1}}%
	{\XXint\textstyle\scriptstyle{#1}}%
	{\XXint\scriptstyle\scriptscriptstyle{#1}}%
	{\XXint\scriptstyle\scriptscriptstyle{#1}}%
	\!\int}
\def\XXint#1#2#3{{\setbox0=\hbox{$#1{#2#3}{\int}$}
		\vcenter{\hbox{$#2#3$}}\kern-.5\wd0}}

\def\Yint#1{\mathchoice
	{\YYint\displaystyle\textstyle{#1}}%
	{\YYint\textstyle\scriptstyle{#1}}%
	{\YYint\scriptstyle\scriptscriptstyle{#1}}%
	{\YYint\scriptscriptstyle\scriptscriptstyle{#1}}%
	\!\iint}
\def\YYint#1#2#3{{\setbox0=\hbox{$#1{#2#3}{\iint}$}
		\vcenter{\hbox{$#2#3$}}\kern-.51\wd0}}
\def\longdash{{-}\mkern-3.5mu{-}} 

\def\fiint{\Yint\longdash}

\def\Xint#1{\mathchoice
	{\XXint\displaystyle\textstyle{#1}}%
	{\XXint\textstyle\scriptstyle{#1}}%
	{\XXint\scriptstyle\scriptscriptstyle{#1}}%
	{\XXint\scriptscriptstyle\scriptscriptstyle{#1}}%
	\!\int}
\def\XXint#1#2#3{{\setbox0=\hbox{$#1{#2#3}{\int}$ }
		\vcenter{\hbox{$#2#3$ }}\kern-.6\wd0}}

\def\dashint{\Xint-}

\usepackage{nameref}
\makeatletter
\let\orgdescriptionlabel\descriptionlabel
\renewcommand*{\descriptionlabel}[1]{%
	\let\orglabel\label
	\let\label\@gobble
	\phantomsection
	\edef\@currentlabel{#1}%
	\let\label\orglabel
	\orgdescriptionlabel{#1}%
}
\makeatother
\numberwithin{equation}{section}
\def\Xint#1{\mathchoice
    {\XXint\displaystyle\textstyle{#1}}%
    {\XXint\textstyle\scriptstyle{#1}}%
    {\XXint\scriptstyle\scriptscriptstyle{#1}}%
    {\XXint\scriptscriptstyle\scriptscriptstyle{#1}}%
    \!\int}
\def\XXint#1#2#3{\setbox0=\hbox{$#1{#2#3}{\int}$}
    \vcenter{\hbox{$#2#3$}}\kern-0.5\wd0}
\def\fint{\Xint-}
\def\dashint{\Xint{\raise4pt\hbox to7pt{\hrulefill}}}

\def\XXiint#1#2#3{\setbox0=\hbox{$#1{#2#3}{\iint}$}
    \vcenter{\hbox{$#2#3$}}\kern-0.5\wd0}

\begin{document}
	
\title[Calder\'on-Zygmund type estimate for the degenerate system]{Calder\'on-Zygmund type estimate for the parabolic double-phase system}

\everymath{\displaystyle}

\makeatletter
\@namedef{subjclassname@2020}{\textup{2020} Mathematics Subject Classification}
\makeatother

\author{Wontae Kim}
\address[Wontae Kim]{Department of Mathematics, Aalto University, P.O. BOX 11100, 00076 Aalto, Finland}
\email{wontae.kim@aalto.fi}

\begin{abstracts}
\abstractin{english}
    This paper provides a local and global Calder\'on-Zygmund type estimate of a weak solution to the parabolic double-phase system. The proof of local estimate is based on comparison estimates and the scaling invariant property of the parabolic double-phase system in the intrinsic cylinders of the stopping time argument setting. For the proof of the global estimate, we have applied the reflection and approximation techniques.
\end{abstracts}

\maketitle

\section{Introduction}

This paper considers the gradient estimate of the degenerate parabolic double-phase system
\[u_t-\dv(b(z)(|\na u|^{p-2}\na u+a(z)|\na u|^{q-2}\na u))=-\dv(|F|^{p-2}F+a(z)|F|^{q-2}F)\]
in the parabolic cylinder $C_R$ defined in \eqref{domain}. 
The double-phase operator consists of two parts. The first part is the $p$-Laplace part and the second part is the $q$-Laplace part for $2\le p< q$. 
The coefficient $b(\cdot)$ of the double-phase operator is bounded from below and above by positive constants while the coefficient $a(\cdot)$ of the $q$-Laplace part is a non-negative H\"older continuous.

We aim to prove the Calder\'on-Zygmund type estimate of the implication
\[H(z,|F|)\in L^\sigma(C_R) \Rightarrow H(z,|\na u|)\in L^\sigma(C_R) \quad\text{for all}\quad \sigma\in(1,\infty),\]
where $H(z,s)=s^p+a(z)s^q.$
The higher integrability estimate of Theorem~\ref{higher} was proved in \cite{KKM}. It says that there exists $\varepsilon_0>0$ sufficiently close to $0$ such that the local estimate of the above implication holds for $\sigma\in(1, 1+\varepsilon_0]$. The proof there is based on the stopping time argument, the reverse H\"older inequality and the Vitali covering argument by dividing intrinsic geometry into two cases.
In order to prove the local estimate for any $\sigma\in(1+\varepsilon_0,\infty)$, the comparison estimate with weak solutions of the homogeneous parabolic double-phase systems and the local regularity properties of such weak solutions are necessary rather than the reverse H\"older inequality. Proposition~\ref{prop1} and Proposition~\ref{prop2} contain the comparison estimate in each intrinsic geometry.
Moreover, these estimates with the stopping time argument prove the Vitali covering lemma and Theorem~\ref{main theorem}. 

The main idea of comparison estimates is to construct the Dirichlet boundary problems of the parabolic double-phase system and to obtain a sufficiently small energy estimate of $u$ and a constructed weak solution in each intrinsic geometry. Since the existence of the Dirichlet boundary problem of the parabolic double-phase system is incomplete, we assume infimum $a(\cdot)$ is strictly positive in Theorem~\ref{main theorem}. The scaling invariant property in each intrinsic geometry is used to get the quantitative estimate for the regularity properties of constructed weak solutions. These estimates and Theorem~\ref{higher} are applied to make energy estimates smaller and to prove the Vitali covering argument.

The global estimate in Theorem~\ref{main theorem2} is proved by the local estimate in Theorem~\ref{main theorem} and reflection argument. Since the obtained estimate is stable with respect to the value of the infimum of $a(\cdot)$, the global estimate can be extended when the infimum of $a(\cdot)$ is $0$, see Corollary~\ref{main theorem3}.

The regularity properties of the elliptic double-phase problems were established in \cite{MR2076158,MR3348922,MR3360738,MR3294408}.
The generalized double-phase settings also have been introduced and regularity properties have been studied in \cite{MR3931352,MR4467321}. The results of the parabolic double-phase problem are only recent. For applications, we refer to \cite{MR2435185,MR1810360}. The existence result of the parabolic double-phase system has been proved in \cite{MR3985549,MR3532237}. For the gradient regularity properties, the difference quotient method was applied in \cite{MR3532237}, the higher integrability was proved in \cite{KKM} and the Lipschitz truncation method, existence and uniqueness of the Dirichlet boundary problem have been researched in \cite{KKS}. 

The Calder\'on-Zygmund estimate was proved for the elliptic $p$-Laplace system ($a(\cdot)\equiv0$) in \cite{MR1246185,MR722254,MR1720770,MR1486629} and for the parabolic $p$-Laplace system in \cite{MR2286632,MR3396087,MR3035434,MR4201641}. It was also proved in the general structure of the Orlicz setting in \cite{MR3859447,MR4475232}, the parabolic $p(\cdot)$-Laplace system in \cite{MR3293436} and elliptic double phase system in \cite{MR3447716,MR3985927}.

\section{Notation and main results}

\subsection{Notations} 
For $x_0\in\RR^n$, $t_0\in\RR$ and $\rho>0$, we denote the ball and cube as
\begin{align*}
\begin{split}
    &B_\rho(x_0)=\{x\in \RR^n:|x-x_0|<\rho\},\\
    &D_\rho=\{x=(x_1,...,x_n)\in \RR^n: |x_i|<R \quad\text{for all}\quad i\in n\}
\end{split}
\end{align*}
and denote the time interval as
\[I_\rho(t_0)=(t_0-\rho^2,t_0+\rho^2).\]
We omit the center point if it is the origin. The parabolic cylinders are defined as the product of the ball and the time interval or product of the cube and the time interval
\begin{align}\label{domain}
    Q_\rho(z_0)=B_\rho(x_0)\times I_\rho(t_0),\quad C_R=D_R\times I_R
\end{align}
for $z_0=(x_0,t_0)\in \RR^n\times \RR$. 

For non-negative function $a(\cdot): C_R\longrightarrow\RR^+$, we defined the function $H(z,s): C_R\times \RR^+\longrightarrow\RR^+$ as $H(z,s)=s^p+a(z)s^q.$
Throughout this paper, $a(\cdot)$ will be chosen as a coefficient of $q$-Laplace operator in the referenced double-phase system and thus $H(\cdot,\cdot)$ is also used as a fixed notation. 

For a function $f\in L^1(Q_\rho(z_0))$ and a measurable set $E\subset Q_\rho(z_0)$, the integral average of $f$ over $E$ is denoted as
\[f_{E}=\fiint_{E}f\,dz.\]

\subsection{Main results}
This paper is concerned with the parabolic double-phase system
\begin{align}\label{11}
\begin{cases}
    u_t-\dv\left(b(z)\mA(z,\na u)\right)=-\dv \mA(z,F)&\text{in}\quad  C_R,\\
    u=0&\text{on}\quad \pa_p C_R.
\end{cases}
\end{align}
Here $b(\cdot): C_R\longrightarrow\RR^+$ is a non-negative measurable function satisfying the ellipticity condition, that is, there exist positive constants $\nu,L$ such that
\begin{align}\label{12}
    0<\nu\le b(z)\le L<\infty\quad\text{for a.e.}\quad z\in  C_R,
\end{align}
the map $\mA(z,\xi): C_R\times \RR^{Nn}\longrightarrow \RR^{n}$ with $N\ge1$ is the parabolic double-phase operator defined as
\[\mA(z,\xi)=|\xi|^{p-2}\xi+a(z)|\xi|^{q-2}\xi,\]
where $a(\cdot): C_R\longrightarrow \RR^+$ is a non-negative function. The source term $F: C_R\longrightarrow\RR^{Nn}$ is a given vector field satisfying
\begin{align}\label{14}
	\iint_{ C_R}H(z,|F|)\,dz<\infty.
\end{align}
Throughout the paper, we assume that exponent $2\le p<q<\infty$ and non-negative function $a(\cdot)$ satisfy assumptions
\begin{align}\label{15}
	q\le p+\tfrac{2\alpha}{n+2},\quad 0\le a\in C^{\alpha,\alpha/2}( C_R)\quad\text{for some}\quad\alpha\in(0,1].
\end{align}
The condition $a\in C^{\alpha,\alpha/2}( C_R)$ means that $a\in L^{\infty}( C_R)$ and there exists a constant $[a]_{\alpha,\alpha/2; C_R}=[a]_{\alpha}>0$ such that for $(x,y)\in D_R$ and $(t,s)\in (0,T)$,
\begin{align}\label{16}
    |a(x,t)-a(y,t)|\le [a]_{\alpha,\alpha/2; C_R}|x-y|^\alpha,\quad
    |a(x,t)-a(x,s)|\le [a]_{\alpha,\alpha/2; C_R}|t-s|^\frac{\alpha}{2}.
\end{align}
We further assume that $b$ has the following VMO condition
\begin{align}\label{18}
    \lim_{r\to0^+}\sup_{|I|\le r^2}\sup_{x_0\in D_R}\fint_{ I\cap I_R}\fint_{ B_r(x_0)\cap D_R}|b(x,t)-b_{( B_{r}(x_0)\times I)\cap  C_R}|\,dx\,dt=0,
\end{align}
where supremum is taken over all balls $B_r(x_0)\subset \mathbb{R}^n$ with $x_0\in D_R$ and all intervals $I\subset \mathbb{R}$ with its length $|I|$ is less than or equal to $r^2$.
The weak solution to \eqref{11} is defined in the following sense.
\begin{definition}
    A measurable function $u: C_{R}\longrightarrow\RR^N$ such that
    \begin{align*}
    \begin{split}
        &u\in C(I_R;L^2(D_R,\RR^N))\cap L^1(I_R;W_0^{1,1}(D_R,\RR^N))\quad\text{with}\\
        &\qquad\qquad\iint_{ C_{R}}H(z,|\na u|)\,dz<\infty
    \end{split}
    \end{align*}
    is a weak solution to \eqref{11} if for every $\varphi\in C_0^\infty( C_R,\RR^N)$
    \[ \iint_{ C_R}\left(-u\cdot \varphi_t+b(z)\mA(z,\na u)\cdot \na\varphi\right)\,dz=\iint_{ C_R}\mA(z,F)\cdot \na\varphi\,dz.\]
    Moreover, the initial boundary condition holds in the sense that
    \begin{align}\label{111}
        \lim_{h\to0^+}\fint_{-R^2}^{-R^{2}+h}\int_{D_R}|u(x,t)|^2\,dx\,dt=0.
    \end{align}
\end{definition}
To simplify the dependency of constant, we write
\begin{align*}
\begin{split}
     &\data_{g}=n,N,p,q,\alpha,\nu,L,[a]_{\alpha},R,\|H(z,|F|)\|_{1},\\
     &\data=\data_g,\|u\|_{L^\infty(I_R;L^2(D_R))},\|H(z,|\na u|)\|_{1},
\end{split}
\end{align*}
where we also shorten 
$\|\cdot\|_{\sigma}=\|\cdot\|_{L^\sigma( C_R)}$ for $\sigma\in[1,\infty].$
Before we state the main results in this paper, we state the local higher integrability result.
\begin{theorem}[\cite{KKM}, Higher integrability]\label{higher}
    Suppose $0\le \inf_{z\in C_R}a(z)$ and let $u$ be the weak solution to \eqref{11}. Then there exist $\varepsilon_0=\varepsilon_0(\data)\in (0,1)$ and $c=c(\data,\|a\|_{\infty})$ such that for any $Q_{2\rho}(z_0)\subset  C_{R}$ and $\varepsilon\in(0,\varepsilon_0]$ there holds
    \begin{align*}
\begin{split}
    &\fiint_{Q_{2\rho}(z_0)}(H(z,|\na u|))^{1+\varepsilon}\,dz
        \le c\left(\fiint_{Q_{2\rho}(z_0)}H(z,|\na u|)\,dz\right)^{\frac{q\varepsilon}{2}+1}\\
        &\qquad\qquad+c\left(\fiint_{Q_{2\rho}(z_0)}(H(z,|F|))^{1+\varepsilon}\,dz+1\right)^\frac{q}{2}.
\end{split}
\end{align*}
\end{theorem}

Throughout this paper, $\varepsilon_0$ denotes the constant in the above theorem.
We now state the main theorems. The first result is the local estimate.
\begin{theorem}\label{main theorem}
    Suppose $0<\inf_{z\in C_R}a(z)$ and let $u$ be the weak solution to \eqref{11}. Then there exists $\rho_0=\rho_0(\data,\|H(z,|F|)\|_{1+\varepsilon_0},\|a\|_{\infty})\in(0,1)$ such that for any $\sigma\in (1+\varepsilon_0,\infty)$ and $ Q_{2\rho_0}(z_0)\subset  C_{R/2}$, there holds
    \begin{align*}
\begin{split}
    &\fiint_{Q_{\rho}(z_0)}(H(z,|\na u|))^{\sigma}\,dz
        \le c\left(\fiint_{Q_{2\rho}(z_0)}H(z,|\na u|)\,dz\right)^{\frac{q(\sigma-1)}{2}+1}\\
        &\qquad\qquad+c\left(\fiint_{Q_{2\rho}(z_0)}(H(z,|F|))^{\sigma}\,dz+1\right)^\frac{q}{2},
\end{split}
\end{align*}
where $c=c(\data,\|a\|_{\infty},\sigma)$ and $\rho\in(0,\rho_0)$.
\end{theorem}

\begin{remark}
Since Theorem~\ref{main theorem} is local, the estimate holds without assumptions on the boundary is not necessary. In particular, \eqref{111} is not used and we may replace the global VMO condition \eqref{18} with the following local VMO condition
\begin{align}\label{115}
    \lim_{r\to0^+}\sup\limits_{\substack{B_{r}(x_0)\times I_\tau(t_0)\subset  C_R,\\\tau\le r^2}} \fiint_{ B_r(x_0)\times I_\tau(t_0)}|b(x,t)-b_{ B_{r}(x_0)\times I_{\tau}(t_0)}|\,dx\,dt=0.
\end{align} 
We also remark that the assumption $0<\inf_{z\in  C_R}a(z)$ is necessary for comparison estimates. See Section~\ref{sec4} for the detail.
\end{remark}

The next two results are the global estimate. The estimate is deduced from the extension argument using the reflection in \cite[Chapter X]{MR1230384} and \cite{MR1799417}. Note that we may replace the constant dependency $\data$ by $\data_g$ for $\varepsilon_0$ by using the standard energy estimate.
\begin{theorem}\label{main theorem2}
    Suppose $0<\inf_{z\in C_R}a(z)$ and let $u$ be the weak solution to \eqref{11}. Then for any $\sigma\in(1,\infty)$, there holds
    \[\fiint_{ C_R}(H(z,|\na u|))^{\sigma}\,dz
        \le c\left(\fiint_{ C_{R}}(H(z,|F|))^\sigma\,dz+1\right)^{\frac{q}{2}},\]
where $\varepsilon_0=\varepsilon_0(\data_g)\in (0,1)$ and
\begin{align*}
    c=
    \begin{cases}
    c(\data_g,\|a\|_{\infty},\sigma)&\text{if}\quad\sigma\in(1,1+\varepsilon_0],\\
     c(\data_g,\|a\|_{\infty},\sigma,\|H(z,|F|)\|_{1+\varepsilon_0})&\text{if}\quad\sigma\in(1+\varepsilon_0,\infty).
    \end{cases}
\end{align*}
\end{theorem}

The above estimate can be extended when the infimum of $a$ is zero.
\begin{corollary}\label{main theorem3}
    Suppose $\inf_{z\in  C_{R}}a(z)=0$ and let $u$ be the weak solution to \eqref{11}. Then for any $\sigma\in(1,\infty)$, there holds
    \[\fiint_{ C_R}(H(z,|\na u|))^{\sigma}\,dz
        \le c\left(\fiint_{ C_{R}}(H(z,|F|))^\sigma\,dz+1\right)^{\frac{q}{2}},\]
where $\varepsilon_0=\varepsilon_0(\data_g)\in (0,1)$ and
\begin{align*}
    c=
    \begin{cases}
    c(\data_g,\|a\|_{\infty},\sigma)&\text{if}\quad\sigma\in(1,1+\varepsilon_0],\\
     c(\data_g,\|a\|_{\infty},\sigma,\|H(z,|F|)\|_{1+\varepsilon_0})&\text{if}\quad\sigma\in(1+\varepsilon_0,\infty).
    \end{cases}
\end{align*}
\end{corollary}

\section{Comparison estimates}\label{sec4}
    In this section, we assume $0<\inf_{z\in  C_{R}}a(z)$ and provide comparison estimates which are used for the proof of Theorem~\ref{main theorem}. This assumption is required to guarantee the existence of non-zero Dirichlet boundary value problems. Indeed we observe for any $\xi\in \RR^{Nn}$
\[\inf_{z\in  C_{R}}a(z)|\xi|^q\le \mA(z,\xi)\cdot\xi\quad\text{and}\quad|\mA(z,\xi)|\le 2^{q-1}(1+\lVert a\rVert_{\infty})(1+|\xi|^{q-1}),\]
which means $\mA(z,\xi)$ is a $q$-Laplace type operator and corresponding weak solution $u$ to \eqref{11} satisfies
\[u\in C(I_{R};L^2(D_R)\cap L^q(I_R;W_0^{1,q}(D_R,\RR^N)),\quad
    u_t\in L^{q'}(I_{R};W^{-1,q'}(D_R,\RR^N)).\]
The existence result of the parabolic $q$-Laplace type system is applicable in this section.

For the comparison estimates, it is necessary to keep track of the dependency of constants. We will use $\epsilon,\delta,K,\rho_0$ as constants. For $\epsilon\in(0,1)$ will be determined later as $\tfrac{1}{2^{q+3}}$ in \eqref{473}, constants $\delta\in(0,1)$, $K>1$ and $\rho_0>0$ will be chosen in this section. To be specific, $\delta$ will be determined depending on $\data$ and $\epsilon$ while $K$ will be chosen to be
\[K=180(1+[a]_{\alpha})\left(\frac{1}{|B_1|}\iint_{Q_{2\rho_0}(z_0)}\left(H(z,|\na u|)+\delta^{-1}H(z,|F|)\right)\,dz+1\right)^\frac{\alpha}{n+2}.\]
Finally, $\rho_0\in(0,1)$ is determined depending on $\data$, $\|a\|_{\infty}$, $\epsilon$ and $\|H(z,|F|)\|_{1+\varepsilon_0}$ and plays a role in the multiplication of $K$ by $\rho_0$ small enough. The circular logic never appears since $\rho_0$ is determined after $\delta$ is chosen. For each lemma, we take $\delta$ and $\rho_0$ sufficiently small and the constants $\delta$ and $\rho_0$ are the smallest constants among lemmas.
In order to simplify our notation, we denote
\[c(\data_{\delta})=c(\data,K)\quad\text{and}\quad V=9K,\]
where $V$ will be chosen to be a covering constant in the Vitali covering argument. 

Employing the intrinsic geometry approach in \cite{KKM}, we consider the $p$-intrinsic cylinder case and the $(p,q)$-intrinsic cylinder case.

\subsection{$p$-intrinsic case} 

The $p$-intrinsic cylinder is defined as
\[Q_{\rho}^\la(z_0)=B_{\rho}(x_0)\times I_{\rho}^\la(t_0),
    \quad I_{\rho}^\la(t_0)=(t_0-\la^{2-p}\rho^2,t_0+\la^{2-p}\rho^2),\]
with a center point $z_0=(x_0,t_0)\in \RR^n\times \RR$, $\rho>0$ and $\la\ge 1$. 

This subsection aims to prove the following estimates.
\begin{proposition}\label{prop1}
Let $\epsilon>0$ be a fixed constant. There exist $\delta=\delta(\data,\epsilon)\in(0,1)$, $\rho_0=\rho_0(\data, \|a\|_{\infty}, \|H(z,|F|)\|_{1+\varepsilon_0}, \epsilon)\in (0,1)$ such that if there exists an intrinsic cylinder $Q_{16V\rho_w}^{\la_w}(w)\subset Q_{2\rho_0}(z_0)\subset   C_{R/2}$ for some $\la_w>1$ satisfying
\begin{enumerate}[label=(\roman*),series=theoremconditions]
    \item $p$-intrinsic case: $K^2\la_w^p\ge a(w)\la_w^q$,
    \item stopping time argument for $p$-intrinsic cylinder: 
    \begin{enumerate}[label=(\alph*),series=theoremconditions]
        \item $\fiint_{Q_{16V\rho_w}^{\la_w}(w)}\left(H(z,|\na u|)+\delta^{-1}H(z,|F|)\right)\,dz< \la_w^p$,
        \item $\fiint_{Q_{\rho_w}^{\la_w}(w)}\left(H(z,|\na u|)+\delta^{-1}H(z,|F|)\right)\,dz=\la_w^p$,
    \end{enumerate}
\end{enumerate}
then there exists a weak solution $v_w$ to
\[\pa_tv_w-\dv(b_0(|\na v_w|^{p-2}\na v_w+a_s|\na v_w|^{q-2}\na v_w))=0\]
in $Q_{2V\rho_w}^{\la_w}(w)$ such that
\[\iint_{Q_{V\rho_w}^{\la_w}(w)}H(z,|\na u-\na v_w|)\,dz\le \epsilon \la_w^p|Q_{\rho_w}^{\la_w}|\]
and the following local Lipschitz estimate holds
\[\sup_{z\in Q_{V\rho_w}^{\la_w}(w)}|\na v_w(z)|\le c\la_w,\]
where $c=c(\data_\delta)>0$,
\[b_0=b_{Q_{4V\rho_w}^{\la_w}(w)}\quad\text{and}\quad   a_s=\sup_{z\in Q_{2V\rho_w}^{\la_w}(w)}a(z).\]
\end{proposition}

For simplicity, we assume $w=0$ and write $a_0=a(0)$, $\la=\la_w$ and $\rho=\rho_w$. 
We also denote the assumption in the above proposition $K^2\la^p\ge a_0\la^q$,
 \begin{align}\label{36}
     \fiint_{Q_{16V\rho}^\la}\left(H(z,|\na u|)+\delta^{-1}H(z,|F|)\right)\,dz< \la^p
 \end{align}
and
\begin{align}\label{37}
    \fiint_{Q_{\rho}^\la}\left(H(z,|\na u|)+\delta^{-1}H(z,|F|)\right)\,dz=\la^p.
\end{align}

Denoting $u_0=u_{Q_{8V\rho}^\la}$, we first provide $L^\infty-L^2$ and $L^p$ estimates of $u$.
\begin{lemma}\label{lem31}
 There exists $c=c(\data_\delta)$ such that
\[\la^{p-2}\sup_{t\in I_{8V\rho}^{\la}}\fint_{B_{8V\rho}}\frac{|u(x,t)-u_0|^2}{(8V\rho)^2}\,dx+\fiint_{Q_{8V\rho}^\la}\frac{|u-u_0|^p}{(8V\rho)^p}\,dz\le c\la^p.\]
\end{lemma}
\begin{proof}
 The assumptions \eqref{36} gives
 \[\fiint_{Q_{16V\rho}^\la}\left(H(z,|\na u|)+H(z,|F|)\right)\,dz\le \la^p.\]
 Since we have $K^2\la^p\ge a_0\la^q$, the conclusion follows from \cite[Lemma 4.2 and Lemma 5.1]{KKM}.
\end{proof}

The next lemma will be used in this and the next section.
\begin{lemma}\label{lem42}
   Suppose $c=c(\data_\delta,\|a\|_{\infty},\|H(z,|F|)\|_{1+\varepsilon_0})$ is a constant. Then there exists $\rho_0=\rho_0(\data_\delta,\|a\|_{\infty},\|H(z,|F|)\|_{1+\varepsilon_0},\ep)\in(0,1)$ such that 
\[c\rho^\alpha\la^q\le \frac{1}{(2V)^{n+2}2^{2q}3}\epsilon\la^p.\]
 \end{lemma}

\begin{proof}
    From Theorem~\ref{higher}, there exist $c=c(\data,\|a\|_{\infty})$ and $\varepsilon_0=\varepsilon_0(\data)\in (0,1)$ such that 
   \begin{align*}
   \begin{split}
        &\fiint_{ C_{R/2}}(H(z,|\na u|))^{1+\varepsilon_0}\,dz\le c\left(\fiint_{ C_{R}}H(z,|\na u|)\,dz\right)^{1+\frac{q\varepsilon_0}{2}}\\
        &\qquad\qquad+c\left(\fiint_{ C_R}(H(z,|F|))^{1+\varepsilon_0}\,dz\right)^\frac{q}{2}.
   \end{split}
   \end{align*}
   Since we assumed $Q_{16V\rho}^\la\subset  C_{R/2}$, there exists $c=c(\data,\|a\|_{\infty},\|H(z,|F|)\|_{1+\varepsilon_0})$ such that
   \begin{align}\label{313}
       \iint_{Q_{V\rho}^\la}(H(z,|\na u|))^{1+\varepsilon_0}\,dz\le c.
   \end{align}
   By  applying H\"older's inequality to \eqref{37}, it follows that
    \begin{align*}
        \begin{split}
            &\la^p
            =\fiint_{Q_{\rho}^\la}\left(H(z,|\na u|)+\delta^{-1}H(z,|F|)\right)\,dz\\
            &\qquad\le \left(\fiint_{Q_{\rho}^\la}(H(z,|\na u|)+\delta^{-1}H(z,|F|))^{1+\varepsilon_0}\,dz\right)^\frac{1}{1+\varepsilon_0}\\
            &\qquad\qquad\le c(\data_\delta)\left(\fiint_{Q_{V\rho}^\la}(H(z,|\na u|)+H(z,|F|))^{1+\varepsilon_0}\,dz\right)^\frac{1}{1+\varepsilon_0}.
        \end{split}
    \end{align*}
Denoting $\gamma=\tfrac{\alpha p }{n+2}$, there holds
\begin{align*}
\begin{split}
    &\rho^\alpha\la^q=\rho^\alpha \la^{q-\gamma}\la^{\gamma}\\
    &\qquad\le c(\data_\delta)\rho^\alpha\la^{q-\gamma}\left(\fiint_{Q_{V\rho}^\la}(H(z,|\na u|)+H(z,|F|))^{1+\varepsilon_0}\,dz\right)^\frac{\gamma}{p(1+\varepsilon_0)}.
\end{split}
\end{align*}
Since we have
\begin{align*}
    \begin{split}
     &\left(\fiint_{Q_{V\rho}^\la}(H(z,|\na u|)+H(z,|F|))^{1+\varepsilon_0}\,dz\right)^\frac{\gamma}{p(1+\varepsilon_0)}\\
     &\qquad\le c(\data_\delta)(V\rho)^{-\frac{(n+2)\gamma}{p(1+\varepsilon_0)}}\la^{\frac{(p-2)\gamma}{p(1+\varepsilon_0)}}\\
     &\qquad\qquad\times \left(\iint_{Q_{V\rho}^\la}(H(z,|\na u|)+H(z,|F|))^{1+\varepsilon_0}\,dz\right)^\frac{\gamma}{p(1+\varepsilon_0)},
    \end{split}
\end{align*}
\eqref{313} leads to
\[\rho^\alpha\la^q
    \le c(\data_\delta,\|a\|_{\infty},\|H(z,|F|)\|_{1+\varepsilon_0})\rho^{\alpha-\frac{(n+2)\gamma}{p(1+\varepsilon_0)}}\la^{q-\gamma+\frac{(p-2)\gamma}{p(1+\varepsilon_0)}}.\]
    Note that the following inequalities hold
    \[\frac{(n+2)\gamma}{p(1+\varepsilon_0)}=\frac{\alpha}{1+\varepsilon_0}\quad \text{and}\quad q-\gamma+\frac{(p-2)\gamma}{p}=q-\frac{2\gamma}{p}\le p\]
    and we obtain
    \[\rho^\alpha\la^q\le c(\data_\delta,\|a\|_{\infty},\|H(z,|F|)\|_{1+\varepsilon_0})\rho_0^{\frac{\alpha\varepsilon_0}{1+\varepsilon_0}}\la^p.\]
    The conclusion follows by taking $\rho_0$ sufficiently small enough.
\end{proof}

 We now construct suitable homogeneous functions to obtain the comparison estimates. Let $\zeta\in C(I_{8V\rho}^\la;L^2(B_{8V\rho},\RR^N))\cap L^q(I_{8V\rho}^\la;W^{1,q}(B_{8V\rho},\RR^N))$ be the weak solution to
\begin{align}\label{320}
    \begin{cases}
        \zeta_t-\dv (b\mA(z,\na \zeta))=0&\text{in}\quad Q_{8V\rho}^{\la},\\
        \zeta=u-u_0&\text{on}\quad\pa_p Q_{8V\rho}^{\la}.
    \end{cases}
\end{align}
Since $\na u= \na(u-u_0)$ and $\pa_tu=\pa_t(u-u_0)$, it is clear that $u-u_0$ is a weak solution to 
\[\pa_t(u-u_0)-\dv\mA(z,\na u)=-\dv\mA(z,F)\quad\text{in}\quad Q_{8V\rho}^\la.\]
To derive a suitable energy estimate, we use the following mollification in the time variable since the time derivative of a weak solution for the parabolic system has no function representative in general.
For $f\in L^1( C_R)$ and $0<h<T$, we define the Steklov average $f_h(x,t)$ of $f$ for all $0<t<T$ by
\[f_h(x,t)=
	\begin{cases}
		\fint_t^{t+h}f(x,s)\,ds,&\text{if}\quad 0<t<T-h,\\
		0,&\text{if}\quad T-h\le t.
	\end{cases}\]
For the basic properties of the Steklov average, we refer to~\cite{MR1230384}.

\begin{lemma}\label{lem32}
There exist $\delta=\delta(\data,\epsilon)\in(0,1)$ and $\rho_0=\rho_0(\data_\delta,\|H(z,|F|)\|_{1+\varepsilon_0},\epsilon)\in (0,1)$ such that
    \[\frac{1}{|Q_{\rho}^\lambda|}\iint_{Q_{V\rho}^\la}H(z,|\na u-\na \zeta|)\,dz\le \frac{1}{2^{q}3}\epsilon\la^p.\]
Also, there exists $c=c(\data_\delta)$ such that
\[\la^{p-2}\sup_{t\in I_{8V\rho}^\la}\fint_{B_{8V\rho}}\frac{|\zeta|^2(x,t)}{(8V\rho)^2}\,dx+\fiint_{Q_{8V\rho}^\la}\left(\frac{|\zeta|^p}{(8V\rho)^p} + H(z,|\na \zeta|)\right)\,dz\le c\la^p.\]
\end{lemma}
\begin{proof}
    We take arbitrary $\tau_1,\tau_2\in I_{8V\rho}^{\la}$ such that $\tau_1<\tau_2$ and consider small enough $h>0$, which is chosen to be used for the Steklov average, satisfying $\tau_1,\tau_2\in I_{8V\rho-h}^{\la}$. For small enough $\vartheta>0$, let $\zeta_{\tau_1,\tau_2}^{\vartheta}\in W_0^{1,\infty}(I_{8V\rho-h}^{\la})$ be a cut-off function defined as
    \[\zeta_{\tau_1,\tau_2}^\vartheta(t)=
        \begin{cases}
            \tfrac{1}{\vartheta}(t-(\tau_1-\vartheta))&\text{for}\quad \tau_1-\vartheta\le t\le \tau_1,\\
            \qquad1&\text{for}\quad\tau_1\le t\le \tau_2,\\
            1-\tfrac{1}{\vartheta}(t-\tau_2)&\text{for}\quad\tau_2\le t\le \tau_2+\vartheta,\\
            \qquad0&\text{otherwise.}
        \end{cases}\]
    We take $[u-u_0-\zeta]_h\zeta_{\tau_1,\tau_2}^{\vartheta}$ as a test function to
    \[\pa_t[u-u_0-\zeta]_h-\dv[b(\mA(z,\na u)-\mA(z,\na \zeta))]_h=-\dv[\mA(z,F)]_h\]
    in $B_{8V\rho}\times I_{8V\rho-h}^{\la}.$ Since $\na[u-u_0]_h=[\na (u-u_0)]_h=[\na u]_h=\na [u]_h$, we have
    \begin{align*}
        \begin{split}
           \mathrm{I}+\mathrm{II}
           &= \fiint_{Q_{8V\rho}^\la}\pa_t[u-u_0-\zeta]_h\cdot [u-u_0-\zeta]_h\zeta_{\tau_1,\tau_2}^{\vartheta}\,dz\\
           &\qquad+\fiint_{Q_{8V\rho}^\la} [b(\mA(z,\na u)-\mA(z,\na \zeta))]_h\cdot \na [u-\zeta]_h\zeta_{\tau_1,\tau_2}^{\vartheta}\,dz\\
           &\qquad\qquad=\fiint_{Q_{8V\rho}^\la}[\mA(z,F)]\cdot \na [u-\zeta]_h \zeta_{\tau_1,\tau_2}^\vartheta\,dz= \mathrm{III}.
        \end{split}
    \end{align*}
   We estimate each term in the above display. Applying the integration by parts, there holds
   \begin{align*}
       \begin{split}
           &\mathrm{I}
           =\fiint_{Q_{8V\rho}^\la}\left(\pa_t\frac{1}{2}|[u-u_0-\zeta]_h|^2\right)\zeta_{\tau_1,\tau_2}^\vartheta\,dz\\
           &\qquad=\fiint_{Q_{8V\rho}^\la}-\frac{1}{2}|[u-u_0-\zeta]_h|^2\pa_t\zeta_{\tau_1,\tau_2}^\vartheta\,dz\\
           &\qquad\qquad=\frac{-1}{2|I_{8V\rho}^\la|}\fint_{\tau_1-\vartheta}^{\tau_1}\fint_{B_{8V\rho}}|[u-u_0-\zeta]_h|^2\,dz\\
           &\qquad\qquad\qquad+\frac{1}{2|I_{8V\rho}^{\la}|}\fint_{\tau_2}^{\tau_2+\vartheta}\fint_{B_{8V\rho}}|[u-u_0-\zeta]_h|^2\,dz.
       \end{split}
   \end{align*}
    Therefore, using the convergence property of the Steklov average and then letting $\vartheta$ go to $0^+$, we obtain
    \begin{align*}
    \begin{split}
        \lim_{\vartheta\to0^+}\lim_{h\to0^+}\mathrm{I}
        &=-\frac{1}{2|I_{8V\rho}^\la|}\fint_{B_{8V\rho}}|u-u_0-\zeta|(x,\tau_1)\,dx\\
        &\qquad+\frac{1}{2|I_{8V\rho}^{\la}|}\fint_{B_{8V\rho}}|u-u_0-\zeta|^2(x,\tau_2)\,dx.
    \end{split}
    \end{align*}
Again it follows from the convergence property of the Steklov average that 
\[\lim_{h\to0^+}\mathrm{II}
        =\fiint_{Q_{8V\rho}^\la} b(\mA(z,\na u)-\mA(z,\na \zeta))\cdot \na (u-\zeta)\zeta_{\tau_1,\tau_2}^{\vartheta}\,dz.\]
Using \eqref{12} and \cite[Chapter 1, Lemma 4.4]{MR1230384}, there exists $c=c(n,N,p,q,\nu)$ such that
\[\lim_{\vartheta\to0^+}\lim_{h\to0^+}\mathrm{II}
        \geq c\fiint_{Q_{8V\rho}^\la}H(z,|\na u-\na \zeta|)\chi_{\{\tau_1\le t\le \tau_2\}}\,dz.\]
    While Young's inequality gives
    \begin{align*}
    \begin{split}
        \lim_{\vartheta\to0^+}\lim_{h\to0^+}\mathrm{III}
            &\le c\fiint_{Q_{8V\rho}^\la}H(z,|F|)\chi_{\{\tau_1\le t\le \tau_2\}}\,dz\\
            &\qquad+\frac{c}{2}\fiint_{Q_{8V\rho}^\la}H(z,|\na u-\na \zeta|)\chi_{\{\tau_1\le t\le \tau_2\}}\,dz.
    \end{split}
    \end{align*}
Combining estimates of $\mathrm{I}$, $\mathrm{II}$ and $\mathrm{III}$, we get
\begin{align*}
    \begin{split}
        &\frac{1}{|I_{8V\rho}^{\la}|}\fint_{B_{8V\rho}}|u-u_0-\zeta|^2(x,\tau_2)\,dx+\fiint_{Q_{8V\rho}^\la}H(z,|\na u-\na \zeta|)\chi_{\{\tau_1\le t\le \tau_2\}}\,dz\\
        &\qquad\le \frac{c}{|I_{8V\rho}^\la|}\fint_{B_{8V\rho}}|u-u_0-\zeta|^2(x,\tau_1)\,dx+c\fiint_{Q_{8V\rho}^\la}H(z,|F|)\chi_{\{\tau_1\le t\le \tau_2\}}\,dz.
    \end{split}
\end{align*}
Since $\tau_1$ and $\tau_2$ are arbitrary, $u,\zeta\in C(I_{8V\rho}^\la;L^2(B_{8V\rho},\RR^N))$ and $u-u_0\equiv \zeta$ in $B_{4\rho}\times \{t=-\la^{2-p}(8V\rho)^2\}$, letting $\tau_1$ to $-\la^{2-p}(8V\rho)^2$ and $\tau_2$ to $\la^{2-p}(8V\rho)^2$ to have
\[\fiint_{Q_{8V\rho}^\la}H(z,|\na u-\na \zeta|)\,dz\le c\fiint_{Q_{8V\rho}^\la}H(z,|F|)\,dz.\]
Meanwhile, letting $\tau_1$ to $-\la^{2-p}(8V\rho)^2$ and keeping $\tau_2$ be arbitrary, we also get
\[\frac{1}{|I_{8V\rho}^\la|}\fint_{B_{8V\rho}}|u-u_0-\zeta|^2(x,\tau_2)\,dx\le c\fiint_{Q_{8V\rho}^\la}H(z,|F|)\,dz.\]
Therefore, combining these estimates and using \eqref{36}, we obtain
\begin{align}\label{334}
\begin{split}
    &\sup_{t\in I_{8V\rho}^\la}\frac{1}{|I_{8V\rho}^\la|}\fint_{B_{8V\rho}}|u-u_0-\zeta|^2(x,t)\,dx+\fiint_{Q_{8V\rho}^\la}H(z,|\na u-\na \zeta|)\,dz\\
    &\qquad\le c\fiint_{Q_{8V\rho}^\la}H(z,|F|)\,dz\le c\delta\la^p,
\end{split}
\end{align}
where $c=c(n,N,p,q,\nu,L)$.
    In particular, since $\delta\in(0,1)$, we have
    \[\la^{p-2}\sup_{t\in I_{8V\rho}^\la}\fint_{B_{8V\rho}}\frac{|u-u_0-\zeta|^2(x,t)}{(8V\rho)^2}\,dx+\fiint_{Q_{8V\rho}^\la}H(z,|\na u-\na \zeta|)\,dz\le c\la^p.\]
    Applying the triangle inequality and Lemma~\ref{lem31}, we get
    \[\la^{p-2}\sup_{t\in I_{8V\rho}^\la}\fint_{B_{8V\rho}}\frac{|\zeta|^2(x,t)}{(8V\rho)^2}\,dx+\fiint_{Q_{8V\rho}^\la}H(z,|\na \zeta|)\,dz\le c(\data_\delta)\la^p.\]
    Employing the Poincaré inequality in the spatial direction and Lemma~\ref{lem31}, we also have
    \begin{align*}
    \begin{split}
          &\fiint_{Q_{8V\rho}^\la}\frac{|\zeta|^p}{(8V\rho)^p}\,dz
          \le 2^p\fiint_{Q_{8V\rho}^\la}\frac{|\zeta-(u-u_0)|^p}{(8V\rho)^p}\,dz+2^p\fiint_{Q_{8V\rho}^\la}\frac{|u-u_0|^p}{(8V\rho)^p}\,dz\\
          &\qquad\le c(\data_\delta)\left(\fiint_{Q_{8V\rho}^\la}|\na \zeta-\na u|^p\,dz+\la^p\right)\le c(\data_\delta)\la^p.
    \end{split}
    \end{align*}
    The proof of the second estimate is completed. To obtain the first estimate, we further estimate \eqref{334}. With $V=9K$, we write \eqref{334} as
    \[\frac{1}{|Q_{\rho}^\la|}\iint_{Q_{V\rho}^\la}H(z,|\na u-\na \zeta|)\,dz\le cK^{n+2}\delta\la^p.\]
    Recalling $\tfrac{1}{180(1+[a]_\alpha)}K\delta^\frac{1}{n+2}$ is equivalent to
    \[\left(\frac{\delta^\frac{1}{\alpha}}{|B_1|}\iint_{Q_{2\rho_0}(z_0)}H(z,|\na u|)\,dz+\delta^\frac{1}{\alpha}+\delta^{\frac{1-\alpha}{\alpha}}\iint_{Q_{2\rho_0}(z_0)}H(z,|F|)\,dz\right)^\frac{\alpha}{n+2},\]
    we observe
    \begin{align*}
    \begin{split}
        &\frac{1}{180(1+[a]_\alpha)}K\delta^\frac{1}{n+2}\\
        &\qquad\le\left(\frac{\delta^\frac{1}{\alpha}}{|B_1|}\iint_{ C_{R}}H(z,|\na u|)\,dz+\delta^\frac{1}{\alpha}+\delta^{\frac{1-\alpha}{\alpha}}\iint_{ C_R}H(z,|F|)\,dz\right)^\frac{\alpha}{n+2}.
    \end{split}
    \end{align*}
    If $\alpha\in(0,1)$ holds, then we have $cK^{n+2}\delta\le \tfrac{1}{2^q3}\epsilon$ provided  $\delta=\delta(\data,\epsilon)$ is sufficiently small. On the other hand, if $\alpha=1$ holds, then a further estimate is necessary since we have $\delta^{\frac{1-\alpha}{\alpha}}=1$.
    Since H\"older's inequality gives
    \[\iint_{Q_{2\rho_0}(z_0)}H(z,|F|)\,dz
        \le |Q_{2\rho_0}|^\frac{\varepsilon_0}{1+\varepsilon_0}\left(\iint_{Q_{2\rho_0}(z_0)}(H(z,|F|))^{1+\varepsilon_0}\,dz\right)^\frac{1}{1+\varepsilon_0},\]
    it follows
    \begin{align*}
    \begin{split}
        \frac{1}{180(1+[a]_{\alpha})}K\delta^\frac{1}{n+2}
        &\le \left(\frac{\delta^\frac{1}{\alpha}}{|B_1|}\iint_{ C_{R}}H(z,|\na u|)\,dz+\delta^\frac{1}{\alpha}\right.\\
        &\qquad\left.+|Q_{2\rho_0}|^\frac{\varepsilon_0}{1+\varepsilon_0}\left(\iint_{ C_R}(H(z,|F|))^{1+\varepsilon_0}\,dz\right)^\frac{1}{1+\varepsilon_0}\right)^\frac{\alpha}{n+2}.
    \end{split}
    \end{align*}
    Hence by taking sufficiently small $\delta=\delta(\data,\epsilon)$ and $\rho_0=\rho_0(\data,\|H(z,|F|)\|_{1+\varepsilon_0},\epsilon)$, the desired estimate follows.
\end{proof}

The next lemma provides the quantitative estimate of the higher integrability for $|\na \zeta|$ under the setting of the intrinsic cylinder. The constant $\varepsilon_\delta$ in the next lemma depends on $\data_\delta$ and it may be different from $\varepsilon_0(\data)$ in Theorem~\ref{higher}.
\begin{lemma}\label{lem33}
    There exists $\varepsilon_\delta=\varepsilon_\delta(\data_\delta)\in(0,1)$ and $c=c(\data_\delta)$ such that
    \[\fiint_{Q_{4V\rho}^\la}(H(z,|\na \zeta|))^{1+\varepsilon_\delta}\,dz\le c\la^{p(1+\varepsilon_\delta)}.\]
\end{lemma}

\begin{proof}
  For $(x,t)\in Q_{8V}$ and $\xi\in \RR^{Nn}$, we set the scaled functions and maps
  \begin{align*}
  \begin{split}
      &\zeta_{\la}(x,t)=\tfrac{1}{\rho\la}\zeta(\rho x,\la^{2-p}\rho^2t),\\
      &b_\la(x,t)=b(\rho x,\la^{2-p}\rho^2t),\\
      &a_\la(x,t)=\la^{q-p}a(\rho x,\la^{2-p}\rho^2t),\\ 
      &\mA_\la(z,\xi)=|\xi|^{p-2}\xi+a_\la(z)|\xi|^{q-2}\xi,\\
      &H_\la(z,s)=s^p+a_\la(z)s^q.
  \end{split}
    \end{align*}
         Note that $b_\la$ still satisfies the ellipticity condition \eqref{12} in $Q_{8V}$. Moreover, it follows from Lemma~\ref{lem42} that $a_\la\in C^{\alpha,\alpha/2}(Q_{8V})$ with
\begin{align}\label{344}
    [a_{\la}]_{\alpha,\alpha/2;Q_{8V}}= \rho^\alpha \la^{q-p} [a]_{\alpha} \le [a]_{\alpha}.
\end{align}
We claim that $\zeta_\la$ is a weak solution to the type of \eqref{320}.
For any $\varphi_{\la}\in C_0^\infty(Q_{8V},\RR^N)$, there exists $\varphi\in C_0^\infty(Q_{8V\rho}^\la,\RR^N)$ such that $\varphi_{\la}(x,t)=\varphi(\rho x,\la^{2-p}\rho^2t)$ in $Q_{8V}$.
Then we observe that the change of variables gives
\begin{align*}
    \begin{split}
        &-\fiint_{Q_{8V}}\zeta_{\la}(x,t)\cdot \pa_t\varphi_{\la}(x,t)\,dz\\
        &\qquad= -\fiint_{Q_{8V}}\la^{1-p}\rho \zeta(\rho x,\la^{2-p}\rho^2t)\cdot\pa_t\varphi(\rho x,\la^{2-p}\rho^2t)\,dz\\
        &\qquad\qquad= -\fiint_{Q_{8V\rho}^\la}\la^{1-p}\rho \zeta(x,t)\cdot\pa_t\varphi(x,t)\,dz.
    \end{split}
\end{align*}
Recalling $\zeta$ is the weak solution to \eqref{320}, there holds
\begin{align*}
\begin{split}
    &-\fiint_{Q_{8V}}\zeta_{\la}(x,t)\cdot\pa_t\varphi_{\la}(x,t)\,dz\\
    &\qquad=-\fiint_{Q_{8V\rho}^\la}\la^{1-p}\rho b(z)(|\na \zeta(z)|^{p-2}+a(z)|\na \zeta(z)|^{q-2})\na \zeta(z)\cdot \na \varphi(z)\,dz.
\end{split}
\end{align*}
Again applying the change of variables to the last term, we get
\begin{align*}
\begin{split}
    &\fiint_{Q_{8V\rho}^\la}\la^{1-p}\rho b(z)|\na \zeta(z)|^{p-2}\na \zeta(z)\cdot \na \varphi(z)\,dz\\
    &\qquad=\fiint_{Q_{8V}} b(\rho x,\la^{2-p}\rho^2t)\tfrac{|\na \zeta(\rho x,\la^{2-p}\rho^2t)|^{p-2}}{\la^{p-2}}\tfrac{\na \zeta(\rho x,\la^{2-p}\rho^2t)}{\la}\cdot \rho\na \varphi(\rho x,\la^{2-p}\rho^2t)\,dz\\
    &\qquad\qquad=\fiint_{Q_{8V}} b_\la( x,t)|\na \zeta_\la(x,t)|^{p-2}\na \zeta_\la(x,t)\cdot \na \varphi_\la(x,t)\,dz
\end{split}
\end{align*}
and similarly
\begin{align*}
    \begin{split}
        &\fiint_{Q_{8V\rho}^\la}\la^{1-p}\rho b(z)a(z)|\na \zeta(z)|^{q-2}\na \zeta(z)\cdot \na \varphi(z)\,dz\\
        &\qquad=\fiint_{Q_{8V}} b_\la( x,t)a_\la( x,t)|\na \zeta_\la(x,t)|^{q-2}\na \zeta_\la(x,t)\cdot \na \varphi_\la(x,t)\,dz.
    \end{split}
\end{align*}
Therefore, it follows
\[-\fiint_{Q_{8V}}\zeta_{\la}\cdot\pa_t\varphi_{\la}\,dz
        =-\fiint_{Q_{8V}} b_\la(|\na \zeta_\la|^{p-2}+a_{\la}|\na \zeta_{\la}|^{q-2})\na \zeta_{\la}\cdot \na \varphi_{\la}\,dz,\]
or equivalently, $\zeta_\la$ is a weak solution to
\[\pa_t\zeta_\la-\dv(b_\la\mA_{\la}(z,\na \zeta_\la))=0\quad \text{in}\quad Q_{8V}.\]
    Employing Theorem~\ref{higher}, there holds
    \[\fiint_{Q_{4V}}(H_\la(z,|\na \zeta_\la|))^{1+\varepsilon_\delta}\,dz\le c\left(\fiint_{Q_{8V}}H_\la(z,|\na \zeta_\la|)\,dz+1\right)^{\frac{q\varepsilon_\delta}{2}+1},\]
    where $\varepsilon_\delta$ depending on 
    \[n,p,q,\alpha,\nu,L,[a_\la]_{\alpha},8V,\|\zeta_\la\|_{L^\infty(I_{8V};L^2(B_{8V}))},\|H_\la(z,|\na \zeta_\la|)\|_{1}\]
    and $c$ depending on
    \[n,p,q,\alpha,\nu,L,[a_\la]_{\alpha},8V,\|\zeta_\la\|_{L^\infty(I_{8V};L^2(B_{8V}))},\|H_\la(z,|\na \zeta_\la|)\|_{1},\|a_\la\|_{\infty}.\]
    We now investigate the constant dependency of $\varepsilon_\delta$ and $c$ in the above.
    Note that change of variables implies
    \[\fiint_{Q_{8V}}H_\la(z,|\na \zeta_\la|)\,dz
           =\frac{1}{\la^p}\fiint_{Q_{8V\rho}^\la}H(z,|\na \zeta|)\,dz.\]
   Therefore we get from Lemma~\ref{lem32} that
\begin{align*}
\begin{split}
    &\sup_{t\in I_{8V}}\fint_{B_{8V}}|\zeta_\la|^2\,dx+\fiint_{Q_{8V}}H_\la(z,|\na \zeta_\la|)\,dz\\
    &\qquad=\sup_{t\in I_{8V\rho}^\la}\fint_{B_{8V\rho}}\frac{|\zeta|^2}{\la^2\rho^2}\,dx+\frac{1}{\la^p}\fiint_{Q_{8V\rho}^{\la}}H(z,|\na \zeta|)\,dz\le c(\data_\delta).
\end{split}    
\end{align*}
Meanwhile, \eqref{344} and $a_0\la^q\le K^2\la^p$ gives
\begin{align}\label{352}
    \|a_\la\|_{\infty}\le \la^{q-p}a_0+[a_\la]_{\alpha}(8V)^\alpha\le K^2+[a]_{\alpha}(8V)^\alpha\le c(\data_\delta).
\end{align}
Therefore we conclude $\varepsilon_\delta=\varepsilon_\delta(\data_\delta)$ and
\[\fiint_{Q_{4V}}(H_\la(z,|\na \zeta_\la|))^{1+\varepsilon_\delta}\,dz\le c(\data_\delta).\]
    Finally, it comes from the change of variable that
    \[\fiint_{Q_{4V\rho}^\la}(H(z,|\na \zeta|))^{1+\varepsilon_\delta}\,dz\le c(\data_\delta)\la^{p(1+\varepsilon_\delta)}.\]
This completes the proof.
\end{proof}

We next consider the weak solution $\eta\in C(I_{4V\rho}^\la;L^2(B_{4V\rho},\RR^N))\cap L^q(I_{4V\rho}^\la;W^{1,q}(B_{4V\rho},\RR^N))$ to
\[\begin{cases}
        \eta_t-\dv(b_0\mA(z,\na \eta))=0&\text{in}\quad Q_{4V\rho}^\la,\\
        \eta=\zeta&\text{on}\quad\pa_p Q_{4V\rho}^\la,
    \end{cases}\]
where recall $b_0=b_{Q_{4V\rho}^\la}$.
We have the following comparison estimate.
\begin{lemma}\label{lem34}
    There exists $\rho_0=\rho_0(\data_\delta,\epsilon)\in(0,1)$ such that
    \[\frac{1}{|Q_{\rho}^\lambda|}\iint_{Q_{V\rho}^\la}H(z,|\na \zeta-\na \eta|)\,dz\le \frac{1}{2^{2q}3}\epsilon\la^p.\]
    Also, there exists $c=c(\data_\delta)$ such that
    \[\la^{p-2}\sup_{t\in I_{4V\rho}^\la}\fint_{B_{4V\rho}}\frac{|\eta|^2(x,t)}{(4V\rho)^2}\,dx+\fiint_{Q_{4V\rho}^\la}\left(\frac{|\eta|^p}{(4V\rho)^p}+H(z,|\na \eta|)\right)\,dz\le c\la^p.\]
\end{lemma}
\begin{proof}
 Note that $b_0$ still satisfies \eqref{12}. We follow the same argument in the proof of Lemma~\ref{lem32}. For sufficiently small $h,\vartheta>0$ and $\tau_1,\tau_2\in I_{4V\rho-h}^\la$, we take $[\zeta-\eta]_h\zeta_{\tau_1,\tau_2}^{\vartheta}$ as a test function to
    \[\pa_t[\zeta-\eta]_h-\dv[b_0(\mA(z,\na \zeta)-\mA(z,\na \eta))]_h=-\dv[(b_0-b)\mA(z,\na \zeta)]_h\]
    in $B_{4V\rho}\times I_{4V\rho-h}^\la.$ Then there exists a constant $c=c(n,N,p,q,\nu,L)$ such that
    \begin{align}\label{358}
    \begin{split}
        &\la^{p-2}\sup_{t\in I_{4V\rho}^\la}\fint_{B_{4V\rho}}\frac{|\zeta-\eta|^2(x,t)}{(4V\rho)^2}\,dx+\fiint_{Q_{4V\rho}^\la}H(z,|\na \zeta-\na \eta|)\,dz\\
        &\qquad\le c\fiint_{Q_{4V\rho}^\la}|b_0-b(z)||\mA(z,\na \zeta)||\na \zeta-\na \eta|\,dz.
    \end{split}
    \end{align}
To estimate the last term, we apply Young's inequality to have
\begin{align*}
    \begin{split}
        &c\fiint_{Q_{4V\rho}^\la}|b_0-b(z)||\mA(z,\na \zeta)||\na \zeta-\na \eta|\,dz\\
        &\qquad\le c\fiint_{Q_{4V\rho}^\la}|b_0-b(z)||H(z,\na \zeta)|\,dz+\frac{1}{4L}\fiint_{Q_{4V\rho}^\la}|b_0-b(z)|H(z,|\na \zeta-\na \eta|)\,dz\\
        &\qquad\qquad\le c\fiint_{Q_{4V\rho}^\la}|b_0-b(z)||H(z,\na \zeta)|\,dz+\frac{1}{2}\fiint_{Q_{4V\rho}^\la}H(z,|\na \zeta-\na \eta|)\,dz.
    \end{split}
\end{align*}
Therefore, absorbing the last term to the left-hand side of \eqref{358}, it becomes
\begin{align*}
    \begin{split}
        &\la^{p-2}\sup_{t\in I_{4V\rho}^\la}\fint_{B_{4V\rho}}\frac{|\zeta-\eta|^2(x,t)}{(4V\rho)^2}\,dx+\fiint_{Q_{4V\rho}^\la}H(z,|\na \zeta-\na \eta|)\,dz\\
        &\qquad\le c\fiint_{Q_{4V\rho}^\la}|b_0-b(z)||H(z,\na \zeta)|\,dz.
    \end{split}
\end{align*}
We continue to estimate the last term. Applying Hölder's inequality and Lemma~\ref{lem33}, we get
\begin{align*}
    \begin{split}
        &\fiint_{Q_{4V\rho}^\la}|b_0-b(z)||H(z,\na \zeta)|\,dz\\
        &\qquad\le \left(\fiint_{Q_{4V\rho}^\la}|b_0-b(z)|^\frac{1+\varepsilon_\delta}{\varepsilon_\delta}\,dz\right)^\frac{\varepsilon_\delta}{1+\varepsilon_\delta}\left(\fiint_{Q_{4V\rho}^\la}(H(z,\na \zeta))^{1+\varepsilon_\delta}\,dz\right)^\frac{1}{1+\varepsilon_\delta}\\
        &\qquad\qquad\le c(\data_\delta)\left(\fiint_{Q_{4V\rho}^\la}|b_0-b(z)|^{\frac{1+\varepsilon_\delta}{\varepsilon_\delta}}\,dz\right)^\frac{\varepsilon_\delta}{1+\varepsilon_\delta}\la^{p}
    \end{split}
\end{align*}
Recalling \eqref{12}, we have
\begin{align*}
\begin{split}
    &\fiint_{Q_{4V\rho}^\la}|b_0-b(z)|^{\frac{1+\varepsilon_\delta}{\varepsilon_\delta}}\,dz
    \le \fiint_{Q_{4V\rho}^\la}|b_0-b(z)|(|b_0|+|b(z)|)^\frac{1}{\varepsilon_\delta}\,dz\\
    &\qquad\le (2L)^\frac{1}{\varepsilon_\delta}\fiint_{Q_{4V\rho}^\la}|b_0-b(z)|\,dz.
\end{split}
\end{align*}
We apply the local VMO condition in \eqref{115}. Taking $\rho_0=\rho_0(\data_\delta)$ small enough, there holds
\begin{align*}
    \begin{split}
        &\la^{p-2}\sup_{t\in I_{4V\rho}^\la}\fint_{B_{4V\rho}}\frac{|\zeta-\eta|^2(x,t)}{(4V\rho)^2}\,dx+ \fiint_{Q_{4V\rho}^\la}H(z,|\na \zeta-\na \eta|)\,dz\\
        &\qquad\le \frac{1}{(4V)^{n+2}2^{2q}3}\epsilon\la^p.
    \end{split}
\end{align*}
This completes the proof of the first statement of the lemma. The second statement of the lemma also follows from Lemma~\ref{lem32} and the above estimate. We omit the details.
\end{proof}

By considering the scaled map $\eta_\la(x,t)=\tfrac{1}{\rho\la}\eta(\rho x,\la^{2-p}\rho^2t)$ for $(x,t)\in Q_{4V}$ as in the proof of Lemma~\ref{lem33},
we deduce that $\eta_\la$ is a weak solution to
\begin{align}\label{365}
    \pa_t\eta_\la-\dv(b_0\mA_{\la}(z,\na\eta_\la))=0\quad\text{in}\quad Q_{4V}
\end{align}
and there holds
\begin{align}\label{366}
    \sup_{t\in I_{4V}}\fint_{B_{4V}}|\eta_\la(x,t)|^2\,dx+\fiint_{Q_{4V}}\left(|\eta_\la|^p+(H_\la(z,|\na \eta_\la|))^{1+\varepsilon_\delta}\right)\,dz\le c(\data_\delta).
\end{align}
Note that if $q\le p(1+\varepsilon_\delta)$, then we have $|\na \eta_\la|\in L^q(Q_{2V})$ since we have \eqref{366} and $|\na \eta_\la|^{p(1+\varepsilon_\delta)}\le (H_\la(z,|\na \eta_\la|))^{1+\varepsilon_\delta}$. We will prove $|\na \eta_{\la}|\in L^q(Q_{2V})$ for $p(1+\varepsilon_\delta)<q$.
For this, we use the following iteration lemma in \cite[Lemma 8.3]{MR1962933}.
\begin{lemma}\label{lem21}
	Let $0<r<R<\infty$ and $h:[r,R]\longrightarrow\RR^+$ be a non-negative and bounded function. Suppose there exist $\vartheta\in(0,1)$, $A,B\ge0$ and $\gamma>0$ such that
	\[h(r_1)\le \vartheta h(r_2)+\frac{A}{(r_2-r_1)^\gamma}+B
		\quad\text{for all}\quad
		0<r\le r_1<r_2\le R.\]
	Then there exists a constant $c=c(\vartheta,\gamma)$ such that
	\[h(r)\le c\left(\frac{A}{(R-r)^\gamma}+B\right).\]
\end{lemma}

\begin{lemma}\label{lem35}
    There exists $c=c(\data_\delta)$ such that
    \[\fiint_{Q_{2V\rho}^\la}|\na\eta|^q\,dz\le c\la^q.\]
\end{lemma}
\begin{proof}
    We enough to show 
    \[\fiint_{Q_{2V}}|\na \eta_\la|^q\,dz\le c(\data_\delta).\]
    We revisit the proof in \cite[Lemma 4.2]{MR3532237}. In this reference, the above estimate is obtained when $\alpha\in(0,1)$ and $q\in(p,p+\tfrac{2\alpha}{n+2})$. We will modify the proof therein and extend the above estimate to the case when $\alpha\in(0,1]$ and $q\in(p,p+\tfrac{2\alpha}{n+2}]$ by utilizing the higher integrability estimate \eqref{366}. We divide the proof into two cases $\alpha\in(0,1)$ and $\alpha=1$.
    
    \textit{Case $\alpha\in(0,1)$}:
    First of all, we remark that the structural condition in \eqref{365} satisfies the hypothesis there. We recall $\mA_\la(z,\xi)=|\xi|^{p-2}\xi+a_\la(z)|\xi|^{q-2}\xi$ and define $\mathcal{H}_\la(z,|\xi|)=\tfrac{1}{p}|\xi|^p+\tfrac{1}{q}a_\la(z)|\xi|^q.$
    It follows from \eqref{344} and \eqref{352} that there exists $c=c(\data)$ such that
    \begin{align}\label{371}
        \begin{split}
            &\tfrac{1}{c}|\xi|^p\le \mathcal{H}_\la(z,|\xi|)\le c(1+|\xi|)^q,\\
            &|\pa_\xi\mA_\la(z,\xi)|\le c(1+|\xi|)^{q-2},\\
            &|\xi|^{p-2}|\xi'|^2\le c \left(\pa_\xi\mA_\la(z,\xi)\xi'\cdot\xi'\right),\\
            &|\mathcal{H}_\la(x_1,t,\xi)-\mathcal{H}_\la(x_2,t,\xi)|\le c|x_1-x_2|^\alpha(1+|\xi|)^q.
        \end{split}
    \end{align}
    In the reference domains of the proof in \cite[Lemma 4.2]{MR3532237}, we set $Q_{\rho_1}(z_0)=Q_{2V}$, $Q_{\rho_2}(z_0)=Q_{3V}$ and replace $M_{z_0,R}$ by
    \[\sup_{t\in I_{3V}}\int_{B_{3V}}|\eta_\la(x,t)|^2\,dx+\iint_{Q_{3V}}\left(|\eta_\la|^p+|\na \eta_\la|^{p(1+\varepsilon_\delta)}\right)\,dz+1.\]
    Then for any $2V\le r_1<r_2\le 3V$ and $s\in(p,p+\tfrac{2\alpha}{n+2-\alpha})$, there exist $c=c(\data_\delta,s)$ and $\beta=\beta(\data_\delta,s)$ such that
    \[\iint_{Q_{r_1}}|\na\eta_\la|^s\,dz\le \frac{c}{(r_2-r_1)^\beta}\left(\iint_{Q_{r_2}}|\na \eta_\la|^q\,dz+M_{z_0,R}\right)^{1+\frac{s-p}{2}}.\]
    For any $q\in(p,p+\tfrac{2\alpha}{n+2}]$, if $s\in(p+\tfrac{2\alpha}{n+2},p+\tfrac{2\alpha}{n+2-\alpha})$ and $\mu>\tfrac{s}{q}>1$ satisfy
    \begin{align}\label{451}
        \frac{1}{\mu}\left(1+\tfrac{s-p}{2}\right)<1\quad\text{and}\quad\frac{\mu q-s}{\mu -1}\le p(1+\varepsilon_\delta),
    \end{align}
    then using $|\na\eta_\la|^q=|\na\eta_\la|^\frac{s}{\mu}|\na \eta_\la|^{\frac{q\mu-s}{\mu}}$, H\"older's inequality and Young's inequality, we obtain
    \begin{align*}
        \begin{split}
            &\iint_{Q_{r_1}}|\na \eta_\la|^s\,dz
            \le \frac{c}{(r_2-r_1)^\beta}M_{z_0,R}^{1+\frac{s-p}{2}}\\
            &\qquad+\frac{c}{(r_2-r_1)^\beta}\left(\iint_{Q_{r_2}}|\na \eta_\la|^s\,dz\right)^{\frac{1}{\mu}\left(1+\frac{s-p}{2}\right)} \left(\iint_{Q_{r_2}}|\na \eta_\la|^\frac{q\mu-s}{\mu-1}\,dz\right)^{\gamma}\\
            &\qquad\qquad\le \frac{1}{2}\iint_{Q_{r_2}}|\na \eta_\la|^s\,dz+\frac{c}{(r_2-r_1)^{\beta'}}\left(\iint_{Q_{3V}}|\na \eta_\la|^{p(1+\varepsilon_\delta)}\,dz\right)^{\chi}\\
            &\qquad\qquad\qquad+\frac{c}{(r_2-r_1)^{\beta}}M_{z_0,R}^{1+\frac{s-p}{2}}
        \end{split}
    \end{align*}
    for $\gamma=\tfrac{\mu-1}{\mu}\left(1+\tfrac{s-p}{2}\right)$, $c=c(\data_\delta,s,\mu)$, $\chi=\chi(p,s,\mu)$ and $\beta'=\beta'(\beta,p,s,\mu)$. The conclusion follows from Lemma~\ref{lem21} and \eqref{366}. Hence, in the remaining proof we will show that for any $q\in(p,p+\tfrac{2\alpha}{n+2}]$, there exist $s\in(p+\tfrac{2\alpha}{n+2},p+\tfrac{2\alpha}{n+2-\alpha})$ and $\mu>\tfrac{s}{q}$ satisfying \eqref{451}. 
    We observe \eqref{451} is equivalent to
     \[1+\frac{s-p}{2}<\mu\quad\text{and}\quad\mu\le \frac{s-p(1+\varepsilon_\delta)}{q-p(1+\varepsilon_\delta)}.\]
      Note $\tfrac{s-p(1+\varepsilon_\delta)}{q-p(1+\varepsilon_\delta)}>\tfrac{s}{q}$ holds. We may take $\mu=\tfrac{s-p(1+\varepsilon_\delta)}{q-p(1+\varepsilon_\delta)}$ provided
      \[1+\frac{s-p}{2}< \frac{s-p(1+\varepsilon_\delta)}{q-p(1+\varepsilon_\delta)}\Longleftrightarrow q-p(1+\ep_\delta)<2-\frac{4+2\varepsilon_\delta p}{2+s-p}.\]
      Since $q\le p+\tfrac{2\alpha}{n+2}$, the last inequality holds if
      \[\frac{2\alpha}{n+2}<2-\frac{4}{2+s-p}+\varepsilon_\delta p\left(1-\frac{2}{2+s-p}\right),\]
      Using $p+\tfrac{2\alpha}{n+2}< s$, the above inequality is true when
      \[\frac{2\alpha}{n+2}<2-\frac{4}{2+s-p}+\frac{\varepsilon_\delta \alpha p}{n+2+\alpha}.      \]
      Substituting $s=p+\tfrac{2\alpha}{n+2-\alpha}-\kappa$ for sufficient small $\kappa$, the above inequality becomes
      \[\frac{\alpha}{n+2}<1-\frac{1}{1+\frac{\alpha}{n+2-\alpha}-\frac{\kappa}{2}}+\frac{\varepsilon_\delta\alpha p}{2(n+2+\alpha)}\]
       Replacing $\tfrac{n+2-\alpha}{2}\kappa$ by $\kappa$, there holds
       \[\frac{\alpha}{n+2}<1-\frac{n+2-\alpha}{n+2-\kappa}+\frac{\varepsilon_\delta\alpha p}{2(n+2+\alpha)}.\]
       Finally, the above inequality is equivalent to
       \[\frac{n+2-\alpha}{n+2-\kappa}<\frac{n+2-\alpha}{n+2}+\frac{\varepsilon_\delta \alpha p}{2(n+2+\alpha)}.\]
       Therefore, taking $\kappa=\kappa(\data_\delta)\in(0,1)$ sufficiently small, it is possible to take $\mu=\tfrac{s-p(1+\varepsilon_\delta)}{q-p(1+\varepsilon_\delta)}$. This completes the proof when $\alpha\in(0,1)$.
     
     \textit{Case $\alpha=1$}: In this case, we take $\alpha'\in(0,1)$ and $\kappa\in(0,1)$ satisfying
     \begin{align}\label{464}
         p+\frac{2}{n+2}<p+\frac{2\alpha'}{n+2-\alpha'}
     \end{align}
     and
     \begin{align}\label{386}
         \frac{n+2-\alpha'}{n+2-\kappa}<\frac{n+1}{n+2}+\frac{\varepsilon_\delta p}{2(n+3)}.
     \end{align}
     Note that the above inequalities hold for $\kappa\in(0,1)$ sufficiently small and $\alpha'\in(0,1)$ sufficiently close to $1$ depending on $n,p,\varepsilon_\delta$. Indeed, note that \eqref{386} is equivalent to
     \[n+2-\alpha'<n+1+\frac{\varepsilon_\delta p(n+2)}{2(n+3)}-\kappa\left(\frac{n+1}{n+2}+\frac{\varepsilon_\delta p}{2(n+3)}\right).\]
     Thus we may choose $\kappa< \tfrac{\varepsilon_\delta p(n+2)}{8(n+3)}\left(\tfrac{n+1}{n+2}+\tfrac{\varepsilon_\delta p}{2(n+3)}\right)^{-1}$ and $\alpha'>1-\tfrac{\varepsilon_\delta p(n+2)}{4(n+3)}$.
     Note that the last structure condition in \eqref{371} holds by replacing $\alpha=1$ with $\alpha'\in(0,1)$. Then for any $2V\le r_1<r_2\le 3V$ and $s\in(p,p+\tfrac{2\alpha'}{n+2-\alpha'})$, there exist $c=c(\data_\delta,s)$ and $\beta=\beta(\data_\delta,s)$ such that
    \[\iint_{Q_{r_1}}|\na\eta_\la|^s\,dz\le \frac{c}{(r_2-r_1)^\beta}\left(\iint_{Q_{r_2}}|\na \eta_\la|^q\,dz+M_{z_0,R}\right)^{1+\frac{s-p}{2}}.\]
     Therefore, we enough to show that for any $q\in(p,p+\tfrac{2}{n+2}]$, there exist $s\in(p+\tfrac{2}{n+2},p+\tfrac{2\alpha'}{n+2-\alpha'})$ and $\mu>\tfrac{s}{q}>1$ such that
     \[\frac{1}{\mu}\left(1+\frac{s-p}{2}\right)<1\quad \text{and}\quad\frac{\mu q-s}{\mu-1}\le p(1+\varepsilon_\delta).\]
     Note that $(p+\tfrac{2}{n+2},p+\tfrac{2\alpha'}{n+2-\alpha'})$ is nonempty since \eqref{464} holds. Again, the above inequality is equivalent to
     \[1+\frac{s-p}{2}<\mu\quad \text{and}\quad \mu\le \frac{s-p(1+\varepsilon_\delta)}{q-p(1+\varepsilon_\delta)}.\]
     As in the previous case, we take $\mu=\tfrac{s-p(1+\varepsilon_\delta)}{q-p(1+\varepsilon_\delta)}$ by proving
     \[1+\frac{s-p}{2}<\frac{s-p(1+\varepsilon_\delta)}{q-p(1+\varepsilon_\delta)}
             \Longleftrightarrow q-p(1+\varepsilon_\delta)<2-\frac{4+2\varepsilon_\delta p}{2+s-p}.\]
     Using the fact that $q\le p+\tfrac{2}{n+2}<s$, we suffice to show
     \[\frac{2}{n+2}<2-\frac{4}{2+s-p}+\frac{\varepsilon_\delta p}{n+3}.\]
    Substituting $s=p+\tfrac{2\alpha'-2\kappa}{n+2-\alpha'}$ to the above inequality, it is equivalent to \eqref{386}. This completes the proof.
\end{proof}

Finally, consider the weak solution $v\in C(I_{2V\rho}^\la;L^2(B_{2V\rho},\RR^N))\cap L^q(I_{2V\rho}^\la;W^{1,q}(B_{2V\rho},\RR^N))$ to 
\[\begin{cases}
        v_t-\dv(b_0(|\na v|^{p-2}\na v+a_s|\na v|^{q-2}\na v))=0&\text{in}\quad Q_{2V\rho}^\la,\\
        v=\eta &\text{on}\quad \pa_pQ_{2V\rho}^\la,
    \end{cases}\]
where $a_s=\sup_{z\in Q_{2V\rho}^\la}a(z)$.

\begin{lemma}\label{lem36}
        There holds
       \[\frac{1}{|Q_{\rho}^\lambda|}\iint_{Q_{V\rho}^\la}\left(|\na \eta-\na v|^p+a_s|\na \eta-\na v|^q\right)\,dz\le \frac{1}{2^{2q} 3}\ep\la^p.\]
        Also, there exists $c=c(\data_\delta)$ such that
        \[\fiint_{Q_{2V\rho}^\la}\left(|\na v|^p+a_s|\na v|^q\right)\,dz\le c\la^p.\]
\end{lemma}
 
\begin{proof}
    The proof is similar to the proof of Lemma~\ref{lem34}. There exists $c=c(n,N,p,q,\nu,L)$ such that
    \begin{align*}
    \begin{split}
        &\fiint_{Q_{2V\rho}^\la} \left(|\na \eta-\na v|^p+a_s|\na \eta-\na v|^q\right)\,dz\\
        &\qquad\le c\fiint_{Q_{2V\rho}^\la}|a(z)-a_s||\na \eta|^{q-1}|\na \eta-\na v|\,dz.
    \end{split}
    \end{align*}
     Since there holds $|a(z)-a_s|\le [a]_{\alpha}(2V\rho)^\alpha$
     in $Q_{2V\rho}^\la$, we apply Young's inequality to have
     \begin{align*}
         \begin{split}
             &c\fiint_{Q_{2V\rho}^\la}|a(z)-a_s||\na \eta|^{q-1}|\na \eta-\na v|\,dz\\
             &\qquad\le c\fiint_{Q_{2V\rho}^\la}|a(z)-a_s||\na \eta|^q\,dz+\fiint_{Q_{2V\rho}^\la}\frac{|a(z)-a_s|}{4} |\na\eta-\na v|^q\,dz\\
             &\qquad\qquad\le c(V\rho)^\alpha\fiint_{Q_{2V\rho}^\la}|\na \eta|^q\,dz+\fiint_{Q_{V\rho}^\la}\frac{a_s}{2} |\na\eta-\na v|^q\,dz.
         \end{split}
     \end{align*}
   Therefore we obtain
    \[\fiint_{Q_{2V\rho}^\la}\left(|\na \eta-\na v|^p+a_s|\na \eta-\na v|^q\right)\,dz\le c(V\rho)^\alpha\fiint_{Q_{2V\rho}^\la}|\na \eta|^q\,dz.\]
    For the last term, we use Lemma~\ref{lem35} and Lemma~\ref{lem42} to get
    \[c(V\rho)^\alpha\fiint_{Q_{2V\rho}^\la}|\na \eta|^q\,dz\le c(\data_\delta)\rho^\alpha\la^q\le \frac{1}{(2V)^{n+2}2^{2q}3}\epsilon\la^p.\]
    Therefore it follows that
    \[\frac{1}{|Q_{\rho}^\la|}\iint_{Q_{2V\rho}^\la}H(z,|\na \eta-\na v|)\,dz\le \frac{1}{2^{2q}3}\ep\la^p.\]
    To prove the second statement, we use Lemma~\ref{lem34}, Lemma~\ref{lem35} and the first statement. Then there holds
    \begin{align*}
    \begin{split}
        &\fiint_{Q_{2V\rho}}\left(|\na v|^p+a_s|\na v|^q\right)\,dz\\
        &\qquad\le 2^q\fiint_{Q_{2V\rho}}\left(|\na v-\na \eta|^p+a_s|\na v-\na\eta|^q\right)\,dz+2^q\fiint_{Q_{2V\rho}}\left(|\na \eta|^p+a_s|\na \eta|^q\right)\,dz\\
        &\qquad\qquad\le c(\data_\delta)(\la^p+a_s\la^q).
    \end{split}
    \end{align*}
    Recalling $a_s\le a_0+[a]_\alpha(2V\rho)^\alpha$, it follows from $a_0\la^q\le K^2\la^p$ and Lemma~\ref{lem42} that
    \[\fiint_{Q_{2V\rho}}\left(|\na v|^p+a_s|\na v|^q\right)\,dz\le c(\data_{\delta})\la^p.\]
    This completes the proof.
\end{proof}

The weak solution $v$ satisfies the local Lipschitz regularity in the spatial direction.
\begin{lemma}\label{lem37}
    There exists $c=c(\data_\delta)$ such that
    \[\sup_{z\in Q_{V\rho}^\la}|\na v(z)|\le c\la.\]
\end{lemma}

\begin{proof}
    We again use the scaling argument. For $(x,t)\in Q_{2V}$, let
    \begin{align*}
    \begin{split}
        &v_\la(x,t)=\tfrac{1}{\rho\la}v(\rho x, \la^{2-p}\rho^2 t),\\ &\mathcal{B}_\la(\xi)=b_0(|\xi|^{p-2}\xi+a_s\la^{q-p}|\xi|^{q-2}\xi).
    \end{split}
    \end{align*}
    Then $v_\la$ is a weak solution to
    \[\pa_tv_\la-\dv\mathcal{B}_\la(\na v_\la)=0\quad \text{in}\quad Q_{2V}.\]
    Applying the change of variable to the second inequality in Lemma~\ref{lem36}, we get
    \begin{align}\label{3118}
        \fiint_{Q_{2V}}\left(|\na v_\la|^p+a_s\la^{q-p}|\na v_\la|^q\right)\,dz\le c(\data_\delta).
    \end{align}
    We define a convex function $\varphi_{\la}$ on $[0,\infty)$ as
    \[\varphi_\la(s)=b_0\left(\frac{1}{p}s^p+\frac{1}{q}a_s\la^{q-p}s^q\right).\]
    Then $\varphi_\la\in C^\infty(0,\infty)$ and $\partial_\xi\varphi_\la(|\xi|)=\mathcal{B}_\la(\xi)$ with $\varphi_{\la}(0)=0,$ $\varphi_{\la}'(0)=0,$ $\lim_{s\to\infty}\varphi_{\la}(s)=\infty$.
    Moreover, it is easy to see that there exists a constant $c=c(p,q)$ such that
    $\tfrac{1}{c}s\varphi''_{\la}(s)\le \varphi_{\la}'(s)\le cs\varphi_{\la}''(s)$.
    Therefore applying \cite[Theorem 2.2]{MR3906361}, then there exists $c=c(p,q)$ such that
    \begin{align*}
    \begin{split}
        &\min\biggl\{\sup_{Q_V}(\varphi_\la(|\na v_\la|)^\frac{n}{2}|\na v_\la|^{2-n}),\quad\sup_{Q_V}|\na v_\la|^2\biggr\}\\
        &\qquad\le c\fiint_{Q_{2V}}\left(|\na v_\la|^2+\varphi_{\la}(|\na v_\lambda|)\right)\,dz.
    \end{split}
    \end{align*}
    Observe that if $\sup_{Q_V}|\na v_\la|\le 1$, the there is nothing to prove. Suppose $1\le\sup_{Q_V}|\na v_\la|$. Since we have $p\ge2$ and $\tfrac{\nu}{p}s^p\le \varphi_{\la}(s)$  for $s\ge1$, there holds
    \[\sup_{Q_V}\left(\tfrac{\nu}{p}\right)^\frac{n}{2}|\na v_\lambda|^2\le \sup_{Q_V}\left(\tfrac{\nu}{p}\right)^\frac{n}{2}|\na v_\la|^{\frac{np}{2}+2-n}\le \sup_{Q_V}\varphi_\la(|\na v_\lambda|)^\frac{n}{2}|\na v_\la|^{2-n}.\]
    It follows that
    \begin{align*}
    \begin{split}
        &\min\biggl\{1,\left(\tfrac{\nu}{p}\right)^\frac{n}{2}\biggr\}\sup_{Q_V}|\na v_\la|^2
        \le\min\biggl\{\sup_{Q_V}(\varphi_\la(|\na v_\la|)^\frac{n}{2}|\na v_\la|^{2-n}),\quad\sup_{Q_V}|\na v_\la|^2\biggr\}\\
        &\qquad\le c\fiint_{Q_{2V}}\left(|\na v_\la|^2+\varphi_{\la}(|\na v_\lambda|)\right)\,dz.
    \end{split}
    \end{align*}
    On the other hand recalling $a_s\la^{q-p}\le c(\data_\delta)$, we obtain
    \begin{align*}
      \begin{split}
          &\fiint_{Q_{2V}}|\na v_\la|^2\,dz
          \le \fiint_{Q_{2V}}(|\na v_\la|+1)^p\,dz
          \le \frac{1}{\nu}\fiint_{Q_{2V}}\varphi_\lambda(|\na v_\la|+1)\,dz\\
          &\qquad\le c(\data_\delta)\fiint_{Q_{2V}} \left(\varphi_\lambda(|\na v_\la|)+1\right)\,dz\le c(\data_\delta),
      \end{split}
    \end{align*}
    where we used \eqref{3118} to obtain the last inequality. We conclude
    \[\sup_{z\in Q_{V}}|\na v_{\la}(z)|\le c(\data_\delta).\]
    The proof is completed by applying the change of variables.
\end{proof}
Combining lemmas in this subsection and the estimate below, the conclusion of Proposition~\ref{prop1} is followed.
\begin{align*}
        \begin{split}
            &H(z,|\na u-\na v|)
            \le 2^qH(z,|\na u-\na \zeta|)+2^qH(z,|\na \zeta-\na v|)\\
            &\quad\le 2^qH(z,|\na u-\na \zeta|)+2^{2q}H(z,|\na \zeta-\na \eta|)+2^{2q}H(z,|\na\eta-\na v|).
        \end{split}
\end{align*}

\subsection{$(p,q)$-intrinsic case}
The $(p,q)$-intrinsic cylinder is defined as
\[G_{\rho}^\la(z_0)=B_{\rho}(x_0)\times J_{\rho}^\la(t_0),\quad J_\rho^\la(t_0)=(t_0-\tfrac{\la^2}{ H(z_0,\la)}\rho^2,t_0+\tfrac{\la^2}{H(z_0,\la)}\rho^2),\]
with a center point $z_0=(x_0,t_0)\in \RR^n\times \RR$, $\rho>0$ and $\la\ge 1$.
This subsection considers the comparison estimate in $(p,q)$-intrinsic case $K^2\la_w^p<a(w)\la_w^q$.
\begin{proposition}\label{prop2}
Let $\epsilon>0$ be a fixed constant. There exist $\delta=\delta(\data,\epsilon)\in(0,1)$, $\rho_0=\rho_0(\data, \|a\|_{\infty}, \|H(z,|F|)\|_{1+\varepsilon_0}, \epsilon)\in (0,1)$ such that if there exists an intrinsic cylinder $G_{5V\varrho_w}^{\la_w}(w)\subset Q_{2\rho_0}(z_0)\subset   C_{R/2}$ for some $\la_w>1$ satisfying
\begin{enumerate}[label=(\roman*),series=theoremconditions]
    \item $(p,q)$-intrinsic case: $K^2\la_w^p < a(w)\la_w^q$,
    \item stopping time argument for $(p,q)$-intrinsic cylinder:
    \begin{enumerate}[label=(\alph*),series=theoremconditions]
        \item $\fiint_{G_{5V\varrho_w}^{\la_w}(w)}\left(H(z,|\na u|)+\delta^{-1}H(z,|F|)\right)\,dz< H(w,\la_w)$,
        \item $\fiint_{G_{\varrho_w}^{\la_w}(w)}\left(H(z,|\na u|)+\delta^{-1}H(z,|F|)\right)\,dz=H(w,\la_w)$,
    \end{enumerate}
\end{enumerate}
then there exists a weak solution $v_w$ to
\[\pa_tv_w-\dv(b_0(|\na v_w|^{p-2}\na v_w+a_s|\na v_w|^{q-2}\na v_w))=0\]
in $G_{2V\varrho_w}^{\la_w}(w)$ such that
\[\iint_{G_{V\varrho_w}^{\la_w}(w)}H(z,|\na u-\na v_w|)\,dz\le \epsilon H(w,\la_w)|G_{\varrho_w}^{\la_w}|,\]
and the following local Lipschitz estimate holds
\[\sup_{z\in G_{V\varrho_w}^{\la_w}(w)}|\na v_w(z)|\le c\la_w,\]
where $c=c(n,p,q,\nu,L)>0$,
\[b_0=b_{G_{2V\varrho_w}^{\la_w}(w)}\quad\text{and}\quad   a_s=\sup_{z\in G_{2V\varrho_w}^{\la_w}(w)}a(z).\]
\end{proposition}

Again we may assume $w=0$ and write $a_0=a(0)$, $\la=\la_w$ and $\rho=\varrho_w$ and assumptions in the above proposition by $K^2\la^p\le a_0\la^q,$
\begin{align}\label{3116}     
      \fiint_{G_{5V\rho}^\la}\left(H(z,|\na u|)+\delta^{-1}H(z,|F|)\right)\,dz< H(0,\la)
\end{align}
and
\begin{align}\label{3120}
    \fiint_{G_{\rho}^\la}\left(H(z,|\na u|)+\delta^{-1}H(z,|F|)\right)\,dz= H(0,\la).
\end{align}
We recall $V=9K$ and
  \[K=180(1+[a]_\alpha)\left(\frac{1}{|B_1|}\iint_{Q_{2\rho_0}(z_0)}\left(H(z,|\na u|)+\delta^{-1}H(z,|F|)\right)\,dz+1\right)^\frac{\alpha}{n+2}.\]
The next lemma proves that $a(\cdot)$ is comparable in $Q_{5V\rho}$. As a consequence, \eqref{11} is a type of $(p,q)$-Laplace system in $G_{5V\rho}^\lambda$.
\begin{lemma}\label{lem53}
There holds
 \[[a]_{\alpha}(5V\rho)^\alpha<\inf_{z\in Q_{5V\rho}}a(z).\]
 Moreover, we have $\tfrac{a(z)}{2}\le a(0)\le 2a(z)$ for all $z\in Q_{5V\rho}$.
\end{lemma}
\begin{proof}
  We prove the first statement by using contradiction.
  Assume on the contrary, then
  \begin{align}\label{521}
      \inf_{w\in Q_{5V\rho}}a(z)\le [a]_{\alpha}(5V\rho)^\alpha=[a]_\alpha(45K\rho)^\alpha\le 45[a]_\alpha \rho^\alpha K.
  \end{align}
  We observe
  \begin{align}\label{547}
      a_0\le \sup_{w\in Q_{5V\rho}}a(z)\le \inf_{w\in Q_{5V\rho}}a(z)+[a]_{\alpha}(5V\rho)^\alpha \le 90[a]_\alpha \rho^\alpha K.
  \end{align}
   On the other hand, it follows from \eqref{3120} and $H(0,\la)<2a_0\la^q$ that 
   \begin{align*}
       \begin{split}
           &a_0\la^q
           \le \fiint_{G_{\rho}^\la}\left(H(z,|\na u|)+\delta^{-1}H(z,|F|)\right)\,dz\\
           &\qquad=\frac{H(0,\la)}{2|B_1|\la^2}\rho^{-(n+2)}\iint_{G_{\rho}^\la}\left(H(z,|\na u|)+\delta^{-1}H(z,|F|)\right)\,dz\\
           &\qquad\qquad\le a_0\la^{q-2}\rho^{-(n+2)}\frac{1}{|B_1|}\iint_{Q_{2\rho_0}(z_0)}\left(H(z,|\na u|)+\delta^{-1}H(z,|F|)\right)\,dz.
       \end{split}
   \end{align*}
   Dividing both sides into $a_0\la^{q-2}\rho^{-(n+2)}$, we have
   \[\rho^{n+2}\la^2\le \frac{1}{|B_1|}\iint_{Q_{2\rho_0}(z_0)}\left(H(z,|\na u|)+\delta^{-1}H(z,|F|)\right)\,dz.\]
   Recalling $K$, we obtain
   \begin{align}\label{3127}
   \begin{split}
       &\rho^{\alpha}\la^\frac{2\alpha}{n+2}
       \le \left(\frac{1}{|B_1|}\iint_{Q_{2\rho_0}(z_0)}\left(H(z,|\na u|)+\delta^{-1}H(z,|F|)\right)\,dz\right)^\frac{\alpha}{n+2}\\
       &\qquad\le \frac{1}{180(1+[a]_{\alpha})}K.
   \end{split}
   \end{align}
    Consequently $K^2\la^p<a_0\la^q$, \eqref{547} and \eqref{3127} lead to
    \[K^2\la^p<a_0\la^q\le 90[a]_\alpha \rho^\alpha K\la^q\le 90[a]_\alpha\rho^\alpha \la^{\frac{2\alpha}{n+2}} K\la^p \le \frac{1}{2}K^{2}\la^p.\]
    It is a contradiction and \eqref{521} is false. The second statement follows from the first statement since
    \[\sup_{z\in Q_{5V\rho}}a(z)\le \inf_{z\in Q_{5V\rho}}a(z)+[a]_{\alpha}(5V\rho)^\alpha\le 2\inf_{z\in Q_{5V\rho}}a(z).\]
    The proof is completed.
\end{proof}

\begin{lemma}\label{lem56}
    Suppose $c=c(\data_\delta,\|a\|_{\infty},\|H(z,|F|)\|_{1+\varepsilon_0})$ is a constant. Then there exists $\rho_0=\rho_0(\data_\delta,\|a\|_{\infty},\|H(z,|F|)\|_{1+\varepsilon_0},\ep)\in(0,1)$ such that 
    \[ c\rho^\alpha\la^q\le \frac{1}{(4V)^{n+2}2^{2q}3}\epsilon\la^p.\]
\end{lemma}

\begin{proof}
The proof is analogous to the proof in Lemma~\ref{lem42}. Recalling $G_{5V\rho}^\la\subset C_{R/2}$ and applying Theorem~\ref{higher}, there exist $c(\data,\|a\|_{\infty},\|H(z,|F|)\|_{1+\varepsilon_0})$ and $\varepsilon_0(\data)\in (0,1)$ such that 
   \[\iint_{G_{V\rho}^\la}(H(z,|\na u|))^{1+\varepsilon_0}\,dz\le c.\]
Therefore we observe from \eqref{3120}, H\"older's inequality and $H(0,\la)<2a_0\la^q$ that
    \begin{align*}
        \begin{split}
            &a_0\la^q\le H(0,\lambda)
            \le \left(\fiint_{G_{\rho}^\la}(H(z,|\na u|)+\delta^{-1}H(z,|F|))^{1+\varepsilon_0}\,dz\right)^\frac{1}{1+\varepsilon_0}\\
            &\qquad\le c(\data_\delta)\left(\fiint_{G_{V\rho}^\la}(H(z,|\na u|)+H(z,|F|))^{1+\varepsilon_0}\,dz\right)^\frac{1}{1+\varepsilon_0}\\
            &\qquad\qquad\le c(\data_\delta,\|a\|_{\infty},\|H(z,|F|)\|_{1+\varepsilon_0})\left(\frac{a_0\la^{q-2}}{|B_1|\rho^{n+2}}\right)^\frac{1}{1+\varepsilon_0}.
        \end{split}
    \end{align*}
    Also, using the fact that $a_0^{-1}\le \la^{q-p}$, we get
    \[\la^q\le c\rho^{-\frac{n+2}{1+\varepsilon_0}}\la^{\frac{\varepsilon_0(q-p)}{1+\varepsilon_0}+\frac{q-2}{1+\varepsilon_0}}=c\rho^{-\frac{n+2}{1+\varepsilon_0}}\la^{\frac{\varepsilon_0(2-p)}{1+\varepsilon_0}+q-2}\]
for $c=c(\data_\delta,\|a\|_{\infty},\|H(z,|F|)\|_{1+\varepsilon_0})$. Setting $\theta=\tfrac{\alpha q }{n+2}$, it follows from the above inequality that
\[\rho^\alpha\la^q
        =\rho^\alpha \la^{\theta}\la^{q-\theta}\le c\rho^{\alpha-\frac{(n+2)\theta }{(1+\varepsilon_0)q}}\la^{\frac{\theta}{q}\left(\frac{\varepsilon_0(2-p)}{1+\varepsilon_0}+q-2\right)+q-\theta}.\]
 Note that
    \[\frac{(n+2)\theta}{(1+\varepsilon_0)q}=\frac{\alpha}{1+\varepsilon_0}\quad\text{and}\quad\frac{\theta(q-2)}{q}+q-\theta=q-\frac{2\alpha}{n+2}\le p.\]
    Thus we conclude
    \[\rho^\alpha\la^q\le c(\data_\delta,\|a\|_{\infty},\|H(z,|F|)\|_{1+\varepsilon_0})\rho_0^{\frac{\alpha\varepsilon_0}{1+\varepsilon_0}}\la^p.\]
    The proof is completed by taking $\rho_0$ sufficiently small.
\end{proof}

We will obtain comparison estimates as in the $p$-intrinsic case. A different scaling argument is required since the scaling factor has changed.

Let $\zeta\in C(J_{4V\rho}^\la;L^2(B_{4V\rho},\RR^N))\cap L^q(J_{4V\rho}^\la;W^{1,q}(B_{4V\rho},\RR^N))$ be the weak solution to
\[\begin{cases}
        \zeta_t-\dv (b\mA(z,\na \zeta))=0&\text{in}\quad G_{4V\rho}^{\la},\\
        \zeta=u&\text{on}\quad\pa_p G_{4V\rho}^{\la}.
    \end{cases}\]

\begin{lemma}\label{lem38}
    There exists $\delta=\delta(\data,\epsilon)$ and $\rho_0=\rho_0(\data_\delta,\|H(z,|F|)\|_{1+\varepsilon_0},\epsilon)$ such that
    \[\frac{1}{|G_{\rho}^\lambda|}\iint_{G_{V\rho}^\la} H(z,|\na u-\na \zeta|)\,dz\le \frac{1}{2^{q} 3}\epsilon H(0,\la).\]
    Also, there exits $c=c(n,N,p,q,\nu,L)$ such that
    \[\fiint_{G_{4V\rho}^\la} H(z,|\na\zeta|)\,dz\le cH(0,\la).\]
\end{lemma}
\begin{proof}
    The proof is analogous to the proof of Lemma~\ref{lem32}, Applying energy estimate and \eqref{3116}, there exists $c=c(n,N,p,q,\nu,L)$ such that
    \begin{align}\label{3149}
        \fiint_{G_{4V\rho}^\la}H(z,|\na u-\na\zeta|)\,dz
        \le c\delta H(0,\lambda)
    \end{align}
   and therefore 
    \[\fiint_{G_{4V\rho}^\la}H(z,|\na \zeta|)\,dz \le cH(0,\lambda).\]
    The proof of the second statement is completed. To prove the first statement, it is necessary to estimate further \eqref{3149}. There holds
    \[\frac{1}{|G_\rho^\lambda|}\iint_{G_{4V\rho}^\la}H(z,|\na u-\na \zeta|)\,dz\le cK^{n+2}\delta H(0,\lambda).\]
    The conclusion follows as in the same argument in the proof of Lemma~\ref{lem32}. Since the calculations are repeated, we omit the details.
\end{proof}

We next consider the weak solution $\eta\in C(J^\la_{4V\rho};L^2(B_{4V\rho},\RR^N))\cap L^q(J_{4V\rho};W^{1,q}(B_{4V\rho},\RR^N))$ to 
\begin{align}\label{3141}
    \begin{cases}
        \eta_t-\dv(b\mA(0,\na \eta))=0&\text{in}\quad G_{4V\rho}^\la,\\
        \eta=\zeta&\text{on}\quad \pa_p G_{4V\rho}^\la.
    \end{cases}
\end{align}

\begin{lemma}\label{lem39}
   There exists $\rho_0=\rho_0(\data_\delta,\ep)\in(0,1)$ such that
   \[\frac{1}{|G_{\rho}^\la|}\iint_{G_{V\rho}^\la} H(z,|\na \zeta-\na \eta|)\,dz\le \frac{1}{2^{2q}3}\epsilon H(0,\la).\]
   Also, there exists $c=c(n,N,p,q,\nu,L)$ such that
   \[\fiint_{G_{4V\rho}^\la}H(0,|\na \eta|)\,dz\le cH(0,\la).\]
\end{lemma}
\begin{proof}
   As in the proof of Lemma~\ref{lem32}, we take $[\zeta-\eta]_h\zeta_{\tau_1,\tau_2}^\vartheta$ as a test function to
   \[\pa_t[\zeta-\eta]_h-\dv[b(\mA(0,\na \zeta)-\mA(0,\na \eta))]_h=-\dv[b(\mA(0,\na \zeta)-\mA(z,\na \zeta))]_h\]
   in $B_{4V\rho}\times J_{4V\rho-h}^\la.$ Then there exists $c=c(n,p,q,\nu,L)$ such that
   \[\fiint_{G_{4V\rho}^\la}H(0,|\na \zeta-\na \eta|)\,dz\le c \fiint_{G_{4V\rho}^\la}|\mA(0,\na \zeta)-\mA(z,\na\zeta)||\na\zeta-\na \eta|\,dz.\]
   Note that Lemma~\ref{lem53} implies
   \begin{align*}
   \begin{split}
       &c|\mA(0,\na \zeta)-\mA(z,\na \zeta)||\na\zeta-\na \eta|=c|a(z)-a(0)||\na \zeta|^{q-1}|\na \zeta-\na \eta|\\
       &\qquad\le c|a(z)-a(0)||\na \zeta|^{q}+\frac{|a(z)-a(0)|}{4}|\na\zeta-\na \eta|^q\\
       &\quad\qquad\le c[a]_{\alpha}(4V\rho)^\alpha|\na \zeta|^{q}+\frac{3a(0)}{4}|\na \zeta-\na \eta|^q.
   \end{split}
   \end{align*}
    Absorbing the second term on the left-hand side of the energy estimate, we obtain
    \begin{align}\label{3159}
        \fiint_{G_{4V\rho}^\la}H(0,|\na \zeta-\na \eta|)\,dz\le c(\data_\delta)\rho^\alpha\fiint_{G_{4V\rho}^\la}|\na \zeta|^q\,dz.
    \end{align}
    Note that it follows from $H(0,\la)<2a(0)\la^q$ and Lemma~\ref{lem38} that 
    \[\fiint_{G_{4V\rho}^\la}a_0|\na \zeta|^q\,dz \le c a_0\la^q.\]
   Dividing both side into $a_0$, it follows that
   \[\fiint_{G_{4V\rho}^\la}|\na \zeta|^q\,dz \le c \la^q.\]
    Substituting the above display to \eqref{3159} and applying Lemma~\ref{lem53} and Lemma~\ref{lem56}, we get
    \[\fiint_{G_{4V\rho}^\la}H(z,|\na \zeta-\na \eta|)\,dz\le \frac{1}{(4V)^{n+2}2^{2q}3}\epsilon\la^p.\]
    Since $\la^p\le H(0,\lambda)$, the first statement of the lemma is proved.
\end{proof}
In what follows, we estimate the higher integrability for $|\na \eta|$. Since the constant depends only on $n,N,q,\nu,L$, without loss of generality, we may write $\varepsilon_0$ as an exponent of the self-improving constant.
\begin{lemma}
There exists $c=c(n,N,p,q,\nu,L)$ and $\varepsilon_0=\varepsilon_0(n,N,p,q,\nu,L)$ such that
\[\fiint_{G_{2V\rho}^\la}(H(0,|\na \eta|))^{1+\varepsilon_0}\,dz\le c(H(0,\la))^{1+\varepsilon_0}.\]
\end{lemma}
\begin{proof}
    We define the scaled functions and maps
    \begin{align*}
    \begin{split}
        &\eta_\la(x,t)=\tfrac{1}{\rho\la}\eta(\rho x,\tfrac{\la^{2}}{H(0,\la)}\rho^2t),\\
        &b_\la(x,t)=b(\rho x,\tfrac{\la^{2}}{H(0,\la)}\rho^2 t),\\
      &\mA_\la(0,\xi)=\tfrac{\la}{H(0,\la)}(\la^{p-1}|\xi|^{p-2}\xi +a_0\la^{q-1}|\xi|^{q-2}\xi)
    \end{split}
    \end{align*}
     for $(x,t)\in Q_{4V}$ and $\xi\in \RR^{Nn}$. Using the fact that $H(0,\la)\le 2a_0\la^q$, Lemma~\ref{lem39} and the change of variables, we have
    \begin{align}\label{3153}
    \begin{split}
        &\fiint_{Q_{4V}}|\na \eta_\la|^q\,dz
        =\fiint_{G_{4V\rho}^\la}\frac{|\na \eta|^q}{\la^q}\,dz\le \fiint_{G_{4V\rho}^\la}\frac{2a_0|\na \eta|^q}{H(0,\la)}\,dz\\
        &\qquad\le \fiint_{G_{4V\rho}^\la}\frac{2H(0,|\na \eta|)}{H(0,\la)}\,dz\le c(n,p,q,\nu,L).
    \end{split}
    \end{align}
     Moreover, the ellipticity condition \eqref{12} for $b_\lambda$ in $Q_{4V}$ holds.
    We claim that $\eta_\la$ is a weak solution to the $q$-Laplace type system.
Let $\varphi_{\la}\in C_0^\infty(Q_{4V},\RR^N)$ be arbitrary and $\varphi\in C_0^\infty(G_{4V\rho}^\la,\RR^N)$ be maps satisfying $\varphi_{\la}(x,t)=\varphi(\rho x,\tfrac{\la^2}{H(0,\lambda)}\rho^2t)$ in $Q_{4V}$.
Applying the change of variables, we obtain
\begin{align*}
    \begin{split}
        &-\fiint_{Q_{4V}}\eta_{\la}(x,t)\cdot\pa_t\varphi_{\la}(x,t)\,dz\\
        &\qquad= -\fiint_{Q_{4V}}\tfrac{\la}{H(0,\lambda)}\rho \eta(\rho x,\tfrac{\la^2}{H(0,\lambda)}\rho^2t)\cdot\pa_t\varphi(\rho x,\tfrac{\la^2}{H(0,\lambda)}\rho^2t)\,dz\\
        &\qquad\qquad= -\fiint_{G_{4V\rho}^\la}\tfrac{\la}{H(0,\lambda)}\rho \eta(x,t)\cdot\pa_t\varphi(x,t)\,dz.
    \end{split}
\end{align*}
Using the fact that $\eta$ is the weak solution to \eqref{3141}, we have
\begin{align*}
\begin{split}
    &-\fiint_{Q_{4V}}\eta_{\la}(x,t)\cdot\pa_t\varphi_{\la}(x,t)\,dz\\
    &\qquad=-\fiint_{G_{4V\rho}^\la}\tfrac{\la}{H(0,\lambda)}\rho b(z)(|\na \eta(z)|^{p-2}+a_0|\na \eta(z)|^{q-2})\na \eta(z)\cdot \na \varphi(z)\,dz.
\end{split}
\end{align*}
Again the change of variables gives
\begin{align*}
\begin{split}
     &-\fiint_{Q_{4V}}\eta_{\la}(x,t)\cdot\pa_t\varphi_{\la}\,dz\\
        &\qquad=-\fiint_{Q_{4V}} \tfrac{\la}{H(0,\lambda)}b_\la(z)(\la^{p-1}|\na \eta_\la|^{p-2}+a_0\la^{q-1}|\na \eta_{\la}(z)|^{q-2})\na \eta_{\la}(z)\cdot \na \varphi_{\la}(z)\,dz.
\end{split}
\end{align*}
We have proved that $\eta_\la$ is a weak solution to
\[\pa_t\eta_\la-\dv(b_\la\mA_{\la}(0,\na \eta_\la))=0\quad\text{in}\quad Q_{4V}.\]
We now investigate the growth condition on $\mathcal{A}_\la(0,\xi)$. Note that
     \[|\mA_\la(0,\xi)|
         \le \tfrac{\la^p}{H(0,\la)}|\xi|^{p-1}+\tfrac{a_0\la^q}{H(0,\la)}|\xi|^{q-1}\le |\xi|^{p-1}+|\xi|^{q-1}\le 2^{q}(|\xi|^{q-1}+1).\]
     For the coercivity of $\mA_\la$, we use the fact that $H(0,\la)\le 2a_0\la^q$ to observe
     \[\mA_\la(0,\xi)\cdot\xi=  \tfrac{\la^p}{H(0,\la)}|\xi|^{p}+\tfrac{a_0\la^q}{H(0,\la)} |\xi|^q\geq \tfrac{a_0\la^q}{H(0,\la)}|\xi|^q\geq\frac{1}{2}|\xi|^q.\]
    Therefore, $\mA_\la(0,\xi)$ is a $q$-Laplacian type operators.
The higher integrability results in \cite{MR1749438} states that there exist $\varepsilon_0=\varepsilon_0(n,N,q,\nu,L)$ and $c=c(n,N,q,\nu,L)$ such that
    \[\fiint_{Q_{2V}}|\na \eta_\lambda|^{q(1+\varepsilon_0)}\,dz\le c\left(\fiint_{Q_{4V}}|\na \eta_\la|^q+1\,dz\right)^{1+\frac{q\varepsilon_0}{2}}.\]
    From \eqref{3153} and change of variable, it follows
    \[\fiint_{G_{2V\rho}^\lambda}|\na \eta|^{q(1+\varepsilon_0)}\,dz\le c\la^{q(1+\varepsilon_0)}.\]
    Moreover, we deduce that
    \begin{align*}
    \begin{split}
       &\fiint_{G_{2V\rho}^\la}(H(0,|\na \eta|))^{1+\varepsilon_0}\,dz
        \le 2^2\fiint_{G_{2V\rho}^\la}\left(|\na \eta_\la|^{p(1+\varepsilon_0)}+a_0^{1+\varepsilon_0}|\na \eta_\la|^{q(1+\varepsilon_0)}\right)\,dz\\
        &\qquad\le 2^2\left(\fiint_{G_{2V\rho}^\la}|\na \eta_\la|^{q(1+\varepsilon_0)}\,dz\right)^\frac{p}{q}+\fiint_{G_{2V\rho}^\la}a_0^{1+\varepsilon_0}|\na \eta_\la|^{q(1+\varepsilon_0)}\,dz\\
        &\quad\qquad\le c(\la^{p(1+\varepsilon_0)}+a_0^{1+\varepsilon_0}\la^{q(1+\varepsilon_0)})\le  c(H(0,\la))^{1+\varepsilon_0}.
    \end{split}
    \end{align*}
    This completes the proof.
\end{proof}
We next consider the weak solution to
\[\begin{cases}
        v_t-\dv(b_0\mA(0,\na v))=0&\text{in}\quad G_{2V\rho}^\la\\
        v=\eta&\text{on}\quad G_{2V\rho}^\la,
    \end{cases}\]
where $b_0=b_{G_{2V\rho}^\la}$. The proof of the following lemma is similar to the proof of Lemma~\ref{lem34}. We omit the detail.
\begin{lemma}
    There exists $\rho_0=\rho_0(n,N,p,q,\nu,L,\epsilon)$ such that
    \[\frac{1}{|G_{\rho}^\lambda|}\iint_{G_{V\rho}^\la}H(z,|\na \eta-\na v|)\,dz\le \frac{1}{2^{2q} 3}\epsilon H(0,\la).\]
    Also, there exists $c=c(n,p,q,\nu)$ such that
    \[\fiint_{G_{2V\rho}^\la}H(0,|\na v|)\,dz\le cH(0,\la).\]
\end{lemma}

    For the Lipschitz regularity for $v$, we again consider a convex function defined as
    \[\varphi_\la(s)=b_0\left(\frac{\la^p}{pH(0,\la)}s^p+\frac{a_0\la^q}{qH(0,\la)}s^q\right).\]
    Then $\varphi\in C^\infty(0,\infty)$, $\varphi_{\la}(0)=0$, $\varphi_{\la}'(0)=0$ and $\lim_{s\to\infty}\varphi_{\la}(s)=\infty$.
    Moreover, it is easy to see that $\pa_{\xi}\varphi_\la(|\xi|)=\mathcal{A}_\la(\xi)$ and there exists a constant $c=c(p,q)$ such that $\tfrac{1}{c}s\varphi''_{\la}(s)\le \varphi_{\la}'(s)\le cs\varphi_{\la}''(s)$.
    Therefore, we obtain the following estimate by applying \cite[Theorem 2.2]{MR3906361}. The proof is similar to the proof of Lemma~\ref{lem37}. We again omit the details.

\begin{lemma}
    There exists $c=c(n,N,p,q,\nu,L)$ such that
    \[\sup_{z\in G_{V\rho}^\la}|\na v(z)|\le c\la.\]
\end{lemma}
The lemmas in this subsection lead to Proposition~\ref{prop2}. We omit the details.

\section{The proof of Theorem~\ref{main theorem}}

Let $\sigma>1+\varepsilon_0$ be a fixed constant in the statement of Theorem~\ref{main theorem}. In this section, we will select $\epsilon=\tfrac{1}{2^{q+3}}$ while $\delta>0$ and $\rho_0>0$ are chosen to satisfy Proposition~\ref{prop1} and Proposition~\ref{prop2}.
\subsection{Stopping time argument}

For $ Q_{2\rho_0}(z_0)\subset  C_{R/2}$ and $\rho\in(0,\rho_0)$, we define
\[\la_0^{2}=\fiint_{Q_{2\rho}(z_0)}\left(H(z,|\na u|)+\delta^{-1}H(z,|F|)+1\right)\,dz\quad\text{and}\quad\La_0=\la_0^p+\sup_{z\in  C_{R}}a(z)\la_0^q.\]
For $r\in(\rho,2\rho)$, we denote the upper-level sets
\begin{align*}
    \begin{split}
        &\Psi(\La,r)=\{z\in Q_{r}(z_0): H(z,|\na u|)>\Lambda\},\\ 
        &\Phi(\La,r)=\{z\in Q_{r}(z_0): H(z,|F|)>\Lambda\}.
    \end{split}
\end{align*}
For $\rho\le r_1<r_2\le 2\rho$, consider
\begin{align}\label{43}
    \La>\left(\frac{32V\rho}{r_2-r_1}\right)^\frac{q(n+2)}{2}\La_0.
\end{align}
Note that $\tfrac{\rho}{r_2-r_1}>1$. For each Lebesgue point $w\in \Psi(\La,r_1)$, there exists $\la_w>1$ such that $\La=\la_w^p+a(w)\la_w^q=H(w,\la_w).$
We claim
\begin{align}\label{55}
    \la_w>\left(\frac{32V\rho}{r_2-r_1}\right)^\frac{n+2}{2}\la_0.
\end{align}
Suppose the above inequality is false then there holds
\[\La=\la_w^p+a(w)\la_w^q\le \left(\frac{32V\rho}{r_2-r_1}\right)^\frac{q(n+2)}{2}(\la_0^p+a(w)\la_0^q)
    \le \left(\frac{32V\rho}{r_2-r_1}\right)^\frac{q(n+2)}{2}\La_0.\]
Since it contradicts to \eqref{43}, \eqref{55} holds. 
The next lemma is the stopping time argument. In the remaining subsection, we fix a Lebesgue point $w\in \Psi(\La,r_1)$.
\begin{lemma}\label{lem51}
There exists $\rho_w\in(0,\tfrac{r_2-r_1}{16V})$ such that
\[\fiint_{Q_{\rho_w}^{\la_w}(w)}\left(H(z,|\na u(z)|)+\delta^{-1}H(z,|F(z)|)\right)\  dz=\la_w^p\]
and for any $r\in(\rho_w,r_2-r_1)$
\[\fiint_{Q_{r}^{\la_w}(w)}\left(H(z,|\na u(z)|)+\delta^{-1}H(z,|F(z)|)\right)\,dz<\la_w^p.\]
Moreover, there holds
\[\la_w\le \left(\frac{2\rho}{\rho_w}\right)^\frac{n+2}{2}\la_0.\]
\end{lemma}
\begin{proof}
 For any $r\in [\tfrac{r_2-r_1}{16V},r_2-r_1)$, we observe $Q_r^{\la_w}(w)\subset Q_{2\rho}(z_0)$ and
\begin{align*}
    \begin{split}
       &\fiint_{Q_{r}^{\la_w}(w)}\left(H(z,|\na u(z)|)+\delta^{-1}H(z,|F(z)|)\right)\,dz\\
       &\qquad\le \la_w^{p-2}\left(\frac{32V\rho}{r_2-r_1}\right)^{n+2}\fiint_{Q_{2\rho}(z_0)}\left(H(z,|\na u(z)|)+\delta^{-1}H(z,|F(z)|)\right)\,dz\\
       &\qquad\qquad\le \la_w^{p-2}\left(\frac{32V\rho}{r_2-r_1}\right)^{n+2}\la_0^2<\la_w^p,
    \end{split}
\end{align*}
where to obtain the last inequality we used \eqref{55}. Note that $w\in \Psi(\La,r_1)$ implies $w\in \Psi(\la_w^p,r_1)$. Since the function
\[r \longrightarrow \fiint_{Q_{r}^{\la_w}(w)}\left(H(z,|\na u(z)|)+\delta^{-1}H(z,|F(z)|)\right) \,dz\]
is continuous, we have
\[\lim_{r\to0^+}\fiint_{Q_{r}^{\la_w}(w)}H(z,|\na u(z)|)\,dz>\la_w^p.\]
Hence, there exists $\rho_w\in(0,\tfrac{r_2-r_1}{16V})$ satisfying the conclusion of the first statement. The last statement in this lemma follows from the first statement and
\begin{align*}
\begin{split}
    &\la_w^p
    \le \la_w^{p-2}\left(\frac{2\rho}{\rho_w}\right)^{n+2}\fiint_{Q_{2\rho}(z_0)}\left(H(z,|\na u(z)|)+\delta^{-1}H(z,|F(z)|)\right)\,dz\\
    &\qquad=\la_w^{p-2}\left(\frac{2\rho}{\rho_w}\right)^{n+2}\la_0^2.
\end{split}
\end{align*}
Dividing both sides into $\la_w^{p-2}$ and then taking $\tfrac{1}{2}$ to the exponent of both sides, the proof of the second statement is completed.
\end{proof}

Note that $Q_{16V\rho_w}^{\la_w}(w)\subset Q_{2\rho_0}(z_0)$. If $p$-intrinsic case ($a(w)\la_w^q\le K^2\la_w^p $) holds, then Lemma~\ref{lem51} satisfies the assumptions in Proposition~\ref{prop1}. 
On the other hand, in the $(p,q)$-intrinsic case ($K^2\la_w^p<a(w)\la_w^q$), the following lemma guarantees the assumptions in Proposition~\ref{prop2}.

\begin{lemma}\label{lem54}
 Suppose $K^2\la_w^p<a(w)\la_w^q$. There exists $\varrho_w\in(0,\rho_w)$ such that
\[\fiint_{G_{\varrho_w}^{\la_w}(w)}\left(H(z,|\na u|)+\delta^{-1}H(z,|F|)\right)\,dz=H(w,\la_w)\]
and for any $r\in(\varrho_w,r_2-r_1)$
\[\fiint_{G_{r}^{\la_w}(w)}\left(H(z,|\na u|)+\delta^{-1}H(z,|F|)\right)\,dz<H(w,\la_w).\]
Moreover, there holds
\[\la_w\le \left(\frac{2\rho}{\varrho_w}\right)^\frac{n+2}{2}\la_0.\]
\end{lemma}
\begin{proof}
Since $a(w)>0$, we see that $\la_w^p<H(w,\la_w)$ and $G_{r}^{\la_w}(w)\subsetneq Q_{r}^{\la_w}(w)$ for all $r>0$. We have from Lemma~\ref{lem51} that for any $\rho\in[\rho_w,r_2-r_1)$
\begin{align*}
\begin{split}
    &\fiint_{G_{r}^{\la_w}(w)}\left(H(z,|\na u|)+\delta^{-1}H(z,|F|)\right)\,dz\\
    &\qquad<\frac{|Q_{r}^{\la_w}|}{|G_{r}^{\la_w}|}\fiint_{Q_{r}^{\la_w}(w)}\left(H(z,|\na u|)+\delta^{-1}H(z,|F|)\right)\,dz\\
    &\qquad\qquad\le \frac{H(w,\la_w)}{\la_w^p} \la_w^p=H(w,\la_w).
\end{split}
\end{align*}
Since $w\in \Psi(\La,r_1)$ and the function
\[r \longrightarrow \fiint_{Q_r^{\la_w}(w)}\left(H(z,|\na u(z)|)+\delta^{-1}H(z,|F(z)|)\right)\,dz\]
is continuous, there exists $\varrho_w\in(0,\rho_w)$ satisfying the first statement of lemma. Meanwhile the second statement follows from 
\begin{align*}
\begin{split}
    &H(w,\la_w)
    \le\frac{H(w,\la_w)}{\la_w^2}\left(\frac{2\rho}{\varrho_w}\right)^{n+2}\fiint_{Q_{2\rho}(z_0)}\left(H(z,|\na u|)+\delta^{-1}H(z,|F|)\right)\,dz\\
    &\qquad=\frac{H(w,\la_w)}{\la_w^2}\left(\frac{2\rho}{\varrho_w}\right)^{n+2}\la_0^2.
\end{split}
\end{align*}
This completes the proof.
\end{proof}

Finally, we end this subsection with the comparability of $\la_{(\cdot)}$. It is necessary for the Vitali covering argument.
Recall from Lemma~\ref{lem42} that $[a]_\alpha(V\rho_w)^\alpha\la_w^q\le \la_w^p$ for $p$-intrinsic case while $[a]_\alpha(V\varrho_w)^\alpha\la_w^q\le \la_w^p$ holds in $(p,q)$-intrinsic case from Lemma~\ref{lem56}.

\begin{lemma}\label{lem52}
    If $K^2\la_w^p\ge a(w)\la_w^q$, then $\la_w\le 2^\frac{1}{p}\la_z$ for any $z\in Q_{V\rho_w}(w)\cap \Psi(\La,r_1)$.
    If $K^2\la_w^p<a(w)\la_w^q$, then the same estimate holds for any $z\in Q_{V\varrho_w}(w)\cap \Psi(\La,r_1)$.
\end{lemma}

\begin{proof}
  We suffice to prove when $K^2\la_w^p\ge a(w)\la_w^q$ since the proof is repeated. The second estimate is proved similarly. We prove it by contradiction. Suppose $\la_z<2^{-\frac{1}{p}}\la_w.$
  Since $z\in Q_{V\rho_w}(w)$, we have $a(z)\le a(w)+[a]_{\alpha}(V\rho_w)^\alpha$ and therefore we get
  \[\La=\la_z^p+a(z)\la_z^q\le \la_z^p+a(w)\la_z^q+[a]_{\alpha}(V\rho_w)^\alpha\la_z^q.\]
  It follows from the assumption $\la_z<2^{-\frac{1}{p}}\la_w$ that
  \[\La< \frac{1}{2}(\la_w^p+a(w)\la_w^q)+\frac{1}{2}[a]_\alpha(V\rho_w)^\alpha\la_w^q\le \frac{1}{2}(\la_w^p+a(w)\la_w^q)+\frac{1}{2}\la_w^p<\La.\]
  It is a contradiction and the proof is completed.
\end{proof}

\subsection{Vitali type covering argument}
For each $z\in \Psi(\La,r_1)$ we denote
\[\mathcal{Q}_z=
    \begin{cases}
    Q_{l_z}^{\la_z}(z) &\text{if}\quad K^2\la_z^p\ge a(z)\la_z^q,\\
    G_{l_z}^{\la_z}(z)&\text{if}\quad K^2\la_z^p<a(z)\la_z^q,
    \end{cases}
     \qquad
     l_z=
    \begin{cases}
    \rho_z&\text{if}\quad K^2\la_z^p\ge a(z)\la_z^q,\\
    \varrho_z&\text{if}\quad K^2\la_z^p<a(z)\la_z^q.
    \end{cases}\]
Consider the family of these intrinsic cylinders $\mathcal{F}=\left\{\mathcal{Q}_z:z\in \Psi(\La,r_1) \right\}$.
Recalling $l_z\le \tfrac{r_2-r_1}{16V}$, we define the subfamily
\[\mathcal{F}_j=\left\{\mathcal{Q}_z\in \mathcal{F}: \frac{r_2-r_1}{16V2^j}<l_z\le \frac{r_2-r_1}{16V2^{j-1}} \right\}\]
for $j\in\mathbb{N}$. We choose $\mathcal{G}_j\subset \mathcal{F}_j$ inductively as follows. We first take $\mathcal{G}_1$ as a maximal disjoint collection of cylinders in $\mathcal{F}_1$. Then each cylinder in $\mathcal{G}_1$ is bounded below and thus $\mathcal{G}_1$ is finite. Indeed, the radius of the cylinder is bounded below from the construction while the scaling factor in time $\la_z^{2-p}$ or $\la_z^2/H(z,\la_z)$ is uniformly bounded below by Lemma~\ref{lem51} and Lemma~\ref{lem54}. For selected $\mathcal{G}_1,...,\mathcal{G}_j$, we select a maximal disjoint subset
\[\mathcal{G}_{j+1}=\left\{ \mathcal{Q}_w\in\mathcal{F}_{j+1}: \mathcal{Q}_w\cap \mathcal{Q}_z=\emptyset\quad\text{for all}\quad\mathcal{Q}_z\in \cup_{k=1}^{j}\mathcal{G}_k\right\}.\]
In the same reasoning $\mathcal{G}_{j}$ are finite and $\mathcal{G}=\cup_{j=1}^\infty \mathcal{G}_j$ is a countable subset of pairwise disjoint cylinders in $\mathcal{F}$. In the remaining of this subsection, we will show that for any $\mathcal{Q}_z\in \mathcal{F}$, there exists $\mathcal{Q}_w\in \mathcal{G}$ such that
\begin{align}\label{541}
    \mathcal{Q}_z\cap\mathcal{Q}_w\ne\emptyset\quad\text{and}\quad\mathcal{Q}_z\subset V\mathcal{Q}_w,
\end{align}
where for any $\kappa>0$, we denoted
\[\kappa\mathcal{Q}_w=
    \begin{cases}
    Q_{\kappa l_w}^{\la_w}(w) &\text{if}\quad K^2\la_w^p\ge a(w)\la_w^q,\\
    G_{\kappa l_w}^{\la_w}(w)&\text{if}\quad K^2\la_w^p<a(w)\la_w^q.
    \end{cases}\]
For each $\mathcal{Q}_z\in\mathcal{F}$, there exists $j\in\mathbb{N}$ such that $\mathcal{Q}_z\in \mathcal{F}_j$. From the maximal disjointedness of $\mathcal{G}_j$, we find $\mathcal{Q}_w\in \cup_{k=1}^j\mathcal{G}_k$ such that 
\[\mathcal{Q}_z\cap \mathcal{Q}_w\ne\emptyset\quad\text{and}\quad l_z\le 2l_w.\]
Before we prove the inclusion in \eqref{541}, note that the standard Vitali covering argument with the above display implies
\begin{align}\label{433}
    Q_{l_z}(z)\subset 5Q_{l_w}(w)=Q_{5l_w}(w).
\end{align}
In particular, $B_{l_z}(x)\subset 5B_{l_w}(y)$ holds where $z=(x,t)$ and $w=(y,s)$. 
Recalling $V\ge 5$, we suffice to prove the inclusion of time intervals in \eqref{541}. 
Note we are able to employ Lemma~\ref{lem52} and Lemma~\ref{lem53} owing to \eqref{433}.

\textit{Case 1: $\mathcal{Q}_z=Q_{l_z}^{\la_z}(z)$ and $\mathcal{Q}_w= Q_{l_w}^{\la_w}(w)$.} 
For any $\tau\in I_{l_z}^{\la_z}(t)$, we observe
\[|\tau-s|\le |\tau-t|+|t-s|\le |I_{l_z}^{\la_z}|+\tfrac{1}{2}|I_{l_w}^{\la_w}|\le 2\la_z^{2-p}l_z^2+\la_w^{2-p}l_w^2.\]
For the scaling factors, there holds $\la_w\le 2^\frac{1}{p}\la_z$ from Lemma~\ref{lem52} and for the radii, we use $l_z\le 2l_w$. Then there holds
\begin{align*}
    \begin{split}
        &2\la_z^{2-p}l_z^2+\la_w^{2-p}l_w^2
        \le 2^{1+\frac{p-2}{p}}\la_w^{2-p}(2l_w)^2+\la_w^{2-p}l_w^2\\
        &\qquad\le \la_w^{2-p}(4l_w)^2+\la_w^{2-p}l_w^2\le \la_w^{2-p}(5l_w)^2,
    \end{split}
\end{align*}
where we used $\tfrac{p-2}{p}\le 1$. Since we have $5\le V$, \eqref{541} holds.

\textit{Case 2: $\mathcal{Q}_z=G_{l_z}^{\la_z}(z)$ and $\mathcal{Q}_w= Q_{l_w}^{\la_w}(w)$.} Note that for any $\tau\in I_{l_z}^{\la_z}(t)$ there holds
\[|\tau-s|
        \le |J_{l_z}^{\la_z}|+\tfrac{1}{2}|I_{l_w}^{\la_w}|\le 2\frac{\la_z^2}{\La}l_z^2+\la_w^{2-p}l_w^2\le 2\la_z^{2-p}l_z^2+\la_w^{2-p}l_w^2.\]
Therefore \eqref{541} follows from the same argument in the previous case.

\textit{Case 3: $\mathcal{Q}_z=Q_{l_z}^{\la_z}(z)$ and $\mathcal{Q}_w= G_{l_w}^{\la_w}(w)$.} It follows from Lemma~\ref{lem52} and Lemma~\ref{lem53} that $\la_w\le 2^\frac{1}{p}\la_z$ and $\tfrac{a(w)}{2}\le a(z)\le 2a(w)$.
Also, recalling $a(z)\la_z^q\le K^2\la_z^p$, $\tfrac{q-2}{p}\le 1$ and $H(w,\la_w)\le2a(w)\la_w^q$, we observe
\[\la_z^{2-p}
    =\frac{\la_z^2}{\la_z^p}\le K^2\frac{\la_z^2}{a(z)\la_z^q}\le 2K^2\frac{\la_w^2}{a(z)\la_w^q}\le 4K^2\frac{\la_w^2}{a(w)\la_w^q}\le 8K^2\frac{\la_w^2}{\La}.\]
Therefore the above display and $l_z\le 2l_w$ lead to that for any $\tau\in I_{l_z}^{\la_z}(t)$,
\[|\tau-s|
        \le |I_{l_z}^{\la_z}|+\tfrac{1}{2}|J_{l_w}^{\la_w}|\le 16K^2\frac{\la_w^2}{\La}(2l_w)^2+\frac{\la_w^2}{\La}l_w^2\le \frac{\la_z^2}{\La}((8K+1)l_w)^2.\]
The conclusion follows from the facts $8K+1\le 9K= V$.

\textit{Case 4: $\mathcal{Q}_z=G_{l_z}^{\la_z}(z)$ and $\mathcal{Q}_w= G_{l_w}^{\la_w}(w)$.} Again we have $\la_w\le 2^\frac{1}{p}\la_z$ and $\tfrac{a(w)}{2}\le a(z)\le 2a(w)$. Since $\tfrac{q-2}{p}\le 1$, there holds
\begin{align*}
\begin{split}
    &\frac{\La}{\la_w^2}=\la_w^{p-2}+a(w)\la_w^{q-2}\le 2^{\frac{q-2}{p}}(\la_z^{p-2}+a(w)\la_z^{q-2})\\
    &\qquad\le 4(\la_z^{p-2}+a(z)\la_z^{q-2})=4\frac{\La}{\la_z^2}.
\end{split}
\end{align*}
Therefore for any $\tau\in I_{l_z}^{\la_z}(t)$, there holds
\[|\tau-s|
        \le |J_{l_z}^{\la_z}|+\tfrac{1}{2}|J_{l_w}^{\la_w}|
        \le 8\frac{\la_w^2}{\La}(2l_w)^2+\frac{\la_w^2}{\La}l_w^2\le \frac{\la_w^2}{\La}(9l_w)^2.\]
Again from the fact $9\le V$, \eqref{541} holds.

All the possible cases are covered. We conclude that there exists pairwise disjoint subfamily $\mathcal{G}=\{\mathcal{Q}_i\}_{i\in\mathbb{N}}$ in $\Psi(\La,r_1)$ such that $\Psi(\La,r_1)\subset \cup_{i\in\mathbb{N}}V\mathcal{Q}_i$ where
\[\mathcal{Q}_i=
    \begin{cases}
        Q_{\rho_i}^{\la_i}(w_i)&\text{if}\quad K^2\la_i^p\ge a(w_i)\la_i^q,\\
        G_{\varrho_i}^{\la_i}(w_i)&\text{if}\quad K^2\la_i^p<a(w_i)\la_i^q,
    \end{cases}\]
$\la_i=\la_{w_i}$ and $\rho_i=\rho_{w_i}$ if $K^2\la_i^p\ge a(w_i)\la_i^q$ or $\varrho_i=\varrho_{w_i}$ if $K^2\la_i^p\ge a(w_i)\la_i^q$.

\subsection{Final proof of the gradient estimate}
In the previous sections, we verified the assumptions in Proposition~\ref{prop1} and Proposition~\ref{prop2}. In order to simplify the notion we use $\eta\in(0,1)$ as a constant in this subsection to denote $\eta=\tfrac{1}{4(K^2+1)}.$
We first suppose $\mathcal{Q}_i=Q_{\rho_i}^{\la_i}(w_i)$. Then Lemma~\ref{lem51} implies
\begin{align*}
\begin{split}
    &|\mathcal{Q}_i|
    =\frac{1}{\la_i^p}\iint_{\mathcal{Q}_i}\left(H(z,|\na u|)+\delta^{-1}H(z,|F|)\right)\,dz\\
    &\qquad= \frac{1}{\la_i^p}\iint_{\mathcal{Q}_i\cap \Psi(\eta\La,r_2)^c}H(z,|\na u|)\,dz+\frac{1}{\la_i^p}\iint_{\mathcal{Q}_i\cap \Psi(\eta\La,r_2)}H(z,|\na u|)\,dz\\
    &\qquad\qquad+\frac{1}{\la_i^p}\iint_{\mathcal{Q}_i\cap\Phi(\eta\delta\La,r_2)^c}\delta^{-1}H(z,|F|)\,dz+\frac{1}{\la_i^p}\iint_{\mathcal{Q}_i\cap \Phi(\eta\delta\La,r_2)}\delta^{-1}H(z,|F|)\,dz.
\end{split}
\end{align*}
Since we have $\La=\la_i^p+a(w_i)\la_i^q\le (K^2+1)\la_i^p$, note that
\[\iint_{\mathcal{Q}_i\cap \Psi(\eta\La,r_2)^c}H(z,|\na u|)\,dz\le \iint_{\mathcal{Q}_i\cap \Psi(\eta\La,r_2)^c}\eta\Lambda\,dz\le \frac{1}{4}\la_i^p|\mathcal{Q}_i|.\]
Similarly, there holds
\[\iint_{\mathcal{Q}_i\cap\Phi(\eta\delta\La,r_2)^c}\delta^{-1}H(z,|F|)\,dz\le  \frac{1}{4}\la_i^p|\mathcal{Q}_i|.\]
Therefore, we obtain
\begin{align}\label{446}
    |\mathcal{Q}_i|\le \frac{2}{\la_i^p}\iint_{\mathcal{Q}_i\cap \Psi(\eta\La,r_2)}H(z,|\na u|)\,dz+\frac{2}{\la_i^p}\iint_{\mathcal{Q}_i\cap\Phi(\eta\delta\La,r_2)}\frac{1}{\delta}H(z,|F|)\,dz.
\end{align}

Meanwhile, it follows from Proposition~\ref{prop1} that there exists $\na v_i\in L^\infty(V\mathcal{Q}_i,\RR^{Nn})$ and $S=S(\data_\delta)$ such that
\begin{align}\label{447}
    \iint_{V\mathcal{Q}_i} H(z,|\na u-\na v_i|)\,dz\le \ep\la_i^p|\mathcal{Q}_i|\quad\text{and}\quad\sup_{z\in V\mathcal{Q}_i}|\na v_i(z)|\le \left(\frac{S}{2^{q+3}}\right)^\frac{1}{q}\la_i.
\end{align}
Also, we apply $[a]_\alpha(V\rho_z)^\alpha\la_z^q\le \la_z^p$ to see that for a.e. $z\in V\mathcal{Q}_i$, there holds
\begin{align*}
\begin{split}
    &H(z,|\na v_i(z)|)\le \frac{S}{2^{q+3}}(\la_i^p+a(z)\la_i^q)\\
    &\qquad\le \frac{S}{2^{q+3}}(H(w_i,\la_i)+[a]_{\alpha}(V\rho_i)^\alpha\la_i^q)\le \frac{S}{2^{q+3}}(H(w_i,\la_i)+\la_i^p).
\end{split}
\end{align*}
Thus, it follows $H(z,|\na v_i(z)|)\le \tfrac{S}{2^{q+2}}\La$ for a.e. $z\in V\mathcal{Q}_i$.
We now claim 
\begin{align}\label{450}
    H(z,|\na v_i(z)|)\le H(z,|\na u(z)-\na v_i(z)|)\quad\text{for a.e.}\quad z\in V\mathcal{Q}_i\cap \Psi(S\La,r_1).
\end{align}
Indeed, if the above inequality is false, then there exists $z\in  V\mathcal{Q}_i\cap \Psi(S\La,r_1)$ such that $H(z,|\na v_i(z)|)>H(z,|\na u(z)-\na v_i(z)|)$ and
\begin{align*}
\begin{split}
    &H(z,|\na v_i(z)|)
    \le \frac{1}{2^{q+2}}S\La    \le \frac{1}{2^{q+2}}H(z,|\na u(z)|)\\
    &\qquad\le \frac{2^q}{2^{q+2}}(H(z,|\na u(z)-\na v_i(z)|)+H(z,|\na v_i(z)|))\\
    &\qquad\qquad\le \frac{2^{q+1}}{2^{q+2}}H(z,|\na v_i(z)|)=\frac{1}{2}H(z,|\na v_i(z)|).
\end{split}
\end{align*}
Thus $0=H(z,|v_i(z)|)>H(z,|\na u(z)|)>S\Lambda$ and it is a contradiction. Employing the first inequality in \eqref{447} and \eqref{450}, it follows
\begin{align*}
    \begin{split}
        &\iint_{V\mathcal{Q}_i\cap \Psi(S\La,r_1)}H(z,|\na u|)\,dz\\
        &\qquad\le 2^q\iint_{V\mathcal{Q}_i\cap \Psi(S\La,r_1)}\left(H(z,|\na u-\na v_i|)+ H(z,|\na v_i|)\right)\,dz\\
        &\qquad\qquad\le 2^{q+1}\iint_{V\mathcal{Q}_i\cap \Psi(S\La,r_1)}H(z,|\na u-\na v_i|)\,dz\le 2^{q+1}\epsilon\la_i^p|\mathcal{Q}_i|.
    \end{split}
\end{align*}
Combining the above inequality with \eqref{446}, we have
\begin{align}\label{453}
\begin{split}
    &\iint_{V\mathcal{Q}_i\cap \Psi(S\La,r_1)}H(z,|\na u|)\,dz
    \le 2^{q+2}\epsilon \iint_{\mathcal{Q}_i\cap \Psi(\eta\La,r_2)}H(z,|\na u|)\,dz\\
    &\qquad\qquad+2^{q+2}\epsilon\iint_{\mathcal{Q}_i\cap\Phi(\eta\delta\La,r_2)}\delta^{-1}H(z,|F|)\,dz.
\end{split}
\end{align}

We next consider when $\mathcal{Q}_i=G_{\varrho_i}^{\la_i}(w_i)$. We will obtain the same estimate in \eqref{453}. Using Lemma~\ref{lem54} and $\eta\le \tfrac{1}{4}$, we get 
\[|\mathcal{Q}_i|
        \le \frac{|\mathcal{Q}_i|}{2}+\frac{1}{\Lambda}\iint_{\mathcal{Q}_i\cap \Psi(\eta\La,r_2)}H(z,|\na u|)\,dz+\frac{1}{\Lambda}\iint_{\mathcal{Q}_i\cap \Phi(\eta\delta\La,r_2)}\delta^{-1}H(z,|F|)\,dz.\]
Thus, we have
\[|\mathcal{Q}_i|\le \frac{2}{\Lambda}\iint_{\mathcal{Q}_i\cap \Psi(\eta\La,r_2)}H(z,|\na u|)\,dz+\frac{2}{\Lambda}\iint_{\mathcal{Q}_i\cap \Phi(\eta\delta\La,r_2)}\delta^{-1}H(z,|F|)\,dz.\]

At the same time, Proposition~\ref{prop2} gives that there exists $\na v_i\in L^\infty(V\mathcal{Q}_i,\RR^n)$ and $S=S(n,p,q,\nu,L)$ such that
\[\iint_{V\mathcal{Q}_i}H(z,|\na u-\na v_i|)\,dz\le \epsilon\La|\mathcal{Q}_i|\quad\text{and}\quad\sup_{z\in V\mathcal{Q}_i}|\na v_i(z)|\le \left(\frac{S}{2^{q+3}}\right)^\frac{1}{q}\la_i.\]
Since the comparability of $a(\cdot)$ in Lemma~\ref{lem53} holds, for $z\in V\mathcal{Q}_i$ there holds
\[H(z,|\na v_i(z)|)\le \frac{S}{2^{q+3}}(\la_i^p+a(z)\la_i^q)\le \frac{S}{2^{q+2}}(\la_i^p+a(w_i)\la_i).\]
Therefore, we obtain $H(z,|\na v_i(z)|)\le \tfrac{S}{2^{q+2}}\La$ for a.e. $z\in V\mathcal{Q}_i$
and $H(z,|\na v_i(z)|)\le H(z,|\na u(z)-\na v_i(z)|)$ for a.e. $z\in V\mathcal{Q}_i\cap \Psi(S\La,r_1).$
Hence, we again get
\[\iint_{V\mathcal{Q}_i\cap \Psi(S\La,r_1)}H(z,|\na u|)\,dz\le  2^{q+1}\epsilon\La|\mathcal{Q}_i|\]
and conclude 
\begin{align}\label{461}
\begin{split}
    &\iint_{V\mathcal{Q}_i\cap \Psi(S\La,r_1)}H(z,|\na u|)\,dz
    \le 2^{q+2}\epsilon \iint_{\mathcal{Q}_i\cap \Psi(\eta\La,r_2)}H(z,|\na u|)\,dz\\
    &\qquad\qquad+2^{q+2}\epsilon\iint_{\mathcal{Q}_i\cap\Phi(\eta\delta\La,r_2)}\delta^{-1}H(z,|F|)\,dz.
\end{split}
\end{align}
On the other hand utilizing the Vitali type covering argument, the covering property gives
\[\iint_{\Psi(S\La,r_1)}H(z,|\na u|)\,dz\le \sum_{i\in\mathbb{N}}\iint_{V\mathcal{Q}_i\cap \Psi(S\La,r_1)}H(z,|\na u|)\,dz\]
while the disjointness property implies
\begin{align*}
\begin{split}
    &\sum_{i\in\mathbb{N}}\biggl( \iint_{\mathcal{Q}_i\cap \Psi(\eta\La,r_2)}H(z,|\na u|)\,dz+\iint_{\mathcal{Q}_i\cap\Phi(\eta\delta\La,r_2)}\delta^{-1}H(z,|F|)\,dz\biggr)\\
        &\qquad\le \iint_{ \Psi(\eta\La,r_2)}H(z,|\na u|)\,dz+\iint_{\Phi(\eta\delta\La,r_2)}\delta^{-1}H(z,|F|)\,dz.
\end{split}
\end{align*}
Since the above displays are connected by \eqref{453} and \eqref{461}, we obtain
\begin{align}\label{463}
    \begin{split}
        &\iint_{\Psi(S\La,r_1)}H(z,|\na u|)\,dz
        \le 2^{q+2}\epsilon \iint_{\mathcal{Q}_i\cap \Psi(\eta\La,r_2)}H(z,|\na u|)\,dz\\
        &\qquad\qquad+2^{q+2}\iint_{\Phi(\eta\delta\La,r_2)}\delta^{-1}H(z,|F|)\,dz.
    \end{split}
\end{align}
We continue by considering the following truncated functions and level-set. For $k>0$, let
\begin{align*}
    \begin{split}
        &H(z,|\na u(z)|)_k=\min\{H(z,|\na u(z)|),k\},\\
        &\Psi_k(\La,\rho)=\{z\in Q_{\rho}:H(z,|\na u|)_k>\Lambda\}.
    \end{split}
\end{align*}
Observe that if $\La>k$, then $\Psi_k(\La,\rho)=\emptyset$ and if $\La\le k$, then $\Psi_k(\La,\rho)=\Psi(\La,\rho)$. Therefore, we deduce from \eqref{463} that
\begin{align}\label{465}
    \begin{split}
    &\iint_{\Psi_k(S\La,r_1)}H(z,|\na u|)\,dz
    \le 2^{q+2}\epsilon \iint_{ \Psi_k(\eta\La,r_2)}H(z,|\na u|)\,dz\\
    &\qquad\qquad+2^{q+2}\iint_{\Phi(\eta\delta\La,r_2)}\frac{1}{\delta}H(z,|F|)\,dz.
\end{split}
\end{align}
Denoting $\La_1=\left(\tfrac{32V\rho}{r_2-r_1}\right)^\frac{q(n+2)}{2}\La_0,$
we integrate \eqref{465} over $(\La_1,\infty)$ with respect to $d\La$ to have
\begin{align}\label{467}
    \begin{split}
        &\mathrm{I}
        =\int_{\La_1}^\infty \La^{\sigma-2}\iint_{\Psi_k(S\La,r_1)}H(z,|\na u|)\,dz\,d\La\\
        &\qquad\le 2^{q+2}\epsilon\int_{\La_1}^\infty \La^{\sigma-2}\iint_{\Psi_k(\eta\La,r_2)}H(z,|\na u|)\,dz\,d\La\\
        &\qquad\qquad +2^{q+2}\int_{\La_1}^\infty \La^{\sigma-2}\iint_{\Phi(\eta\delta\La,r_2)}\delta^{-1}H(z,|F|)\,dz\,d\La=\mathrm{II}+\mathrm{III}.
    \end{split}
\end{align}
To estimate $\mathrm{I}$, we apply the Fubini theorem. There holds
\begin{align*}
    \begin{split}
        &\mathrm{I}
        =\iint_{\Psi_k(S\La_1,r_1)}H(z,|\na u|)\int_{S\La_1}^{H(z,|\na u|)_k}\La^{\sigma-2}\,d\La\,dz\\
        &\qquad=\frac{1}{\sigma-1}\iint_{\Psi_k(S\La_1,r_1)}H(z,|\na u|)(H(z,|\na u|)_k)^{\sigma-1}\,dz\\
        &\qquad\qquad-\frac{1}{\sigma-1}(S\La_1)^{\sigma-1}\iint_{\Psi_k(S\La_1,r_1)}H(z,|\na u|)\,dz.
    \end{split}
\end{align*}
Also since the following estimate holds
\begin{align*}
\begin{split}
    &\iint_{Q_{r_1}(z_0)\setminus \Psi_k(S\La_1,r_1)}H(z,|\na u|)(H(z,|\na u|)_k)^{\sigma-1}\,dz\\
    &\qquad\le (S\Lambda_1)^{\sigma-1}\iint_{Q_{r_2}(z_0)}H(z,|\na u|)\,dz,
\end{split}
\end{align*}
we get
\begin{align*}
\begin{split}
    &\mathrm{I}
        \geq\frac{1}{\sigma-1}\iint_{Q_{r_1}(z_0)}H(z,|\na u|)(H(z,|\na u|)_k)^{\sigma-1}\,dz\\
        &\qquad-\frac{2}{\sigma-1}(S\La_1)^{\sigma-1}\iint_{Q_{2\rho}(z_0)}H(z,|\na u|)\,dz.
\end{split}
\end{align*}
Similarly, we obtain
\begin{align*}
    \begin{split}
        &\mathrm{II}
        \le 2^{q+2}\epsilon\frac{1}{\sigma-1}\iint_{\Psi_k(\La_1,r_2)}H(z,|\na u|)(H(z,|\na u|)_k)^{\sigma-1}\,dz\\
        &\qquad\le 2^{q+2}\epsilon\frac{1}{\sigma-1}\iint_{Q_{r_2}(z_0)}H(z,|\na u|)(H(z,|\na u|)_k)^{\sigma-1}\,dz
        \end{split}
\end{align*}
and
\[\mathrm{III}\le 2^{q+2}\frac{\delta^{-1}}{\sigma-1}\iint_{Q_{2\rho}(z_0)}(H(z,|F|))^{\sigma}\,dz.\]
We have estimated \eqref{467} to be
\begin{align*}
    \begin{split}
        &\iint_{Q_{r_1}(z_0)}H(z,|\na u|)(H(z,|\na u|)_k)^{\sigma-1}\,dz\\
        &\qquad\le 2^{q+2}\epsilon\iint_{Q_{r_2}(z_0)}H(z,|\na u|)(H(z,|\na u|)_k)^{\sigma-1}\,dz\\
        &\qquad\qquad+2(S\La_1)^{\sigma-1}\iint_{Q_{2\rho}(z_0)}H(z,|\na u|)\,dz+2^{q+2}\delta^{-1}\iint_{Q_{2\rho}(z_0)}(H(z,|F|))^{\sigma}\,dz.
    \end{split}
\end{align*}
We take $\epsilon=\tfrac{1}{2^{q+3}}$. Then $\delta$ and $K$ are also fixed and thus $c(\data_\delta)=c(\data)$ and $S=S(\data_\delta)=S(\data)$. Consequently, $\rho_0=\rho_0(\data,\| H(z,|F|) \|_{1+\varepsilon_0},\|a\|_{\infty})\in(0,1)$ is fixed as well.
Recalling $\La_1=\left(\tfrac{32V\rho}{r_2-r_1}\right)^\frac{q(n+2)}{2}\La_0$, it follows
\begin{align}\label{473}
    \begin{split}
        &\iint_{Q_{r_1}(z_0)}H(z,|\na u|)(H(z,|\na u|)_k)^{\sigma-1}\,dz\\
        &\qquad\le \frac{1}{2}\iint_{Q_{r_2}(z_0)}H(z,|\na u|)(H(z,|\na u|)_k)^{\sigma-1}\,dz\\
        &\qquad\qquad+c\left(\tfrac{2\rho}{r_2-r_1}\right)^{\beta}\La_0^{\sigma-1}\iint_{Q_{2\rho}(z_0)}H(z,|\na u|)\,dz+c\iint_{Q_{2\rho}(z_0)}(H(z,|F|))^{\sigma}\,dz,
    \end{split}
\end{align}
where $c=c(\data)$ and $\beta=\tfrac{q(n+2)(\sigma-1)}{2}$.
Using Lemma~\ref{lem21} and then letting $k\longrightarrow\infty$, we have
\begin{align*}
\begin{split}
    &\iint_{Q_{\rho}(z_0)}(H(z,|\na u|))^{\sigma}\,dz
    \le c\La_0^{\sigma-1}\iint_{Q_{2\rho}(z_0)}H(z,|\na u|)\,dz\\
    &\qquad\qquad+c\iint_{Q_{2\rho}(z_0)}(H(z,|F|))^{\sigma}\,dz,
\end{split}
\end{align*}
where $c=c(\data,\sigma)$.
Finally, the following estimate holds from the choice of $\La_0$.
\begin{align*}
\begin{split}
    &\fiint_{Q_{2\rho}(z_0)}(H(z,|\na u|))^{\sigma}\,dz
        \le c\left(\fiint_{Q_{2\rho}(z_0)}H(z,|\na u|)\,dz\right)^{\frac{q(\sigma-1)}{2}+1}\\
        &\qquad\qquad+c\left(\fiint_{Q_{2\rho}(z_0)}(H(z,|F|))^{\sigma}\,dz\right)^\frac{q}{2},
\end{split}
\end{align*}
where $c=c(\data,\|a\|_{\infty},\sigma)$. The proof is completed.

\section{The proof of Theorem~\ref{main theorem2}}
In the proof of Theorem~\ref{main theorem}, the construction of the weak solutions of homogeneous Dirichlet boundary value problems and their regularity properties in the spatial direction are necessary. We avoid this difficulty by extending the solution $u$ and data $b(\cdot)$, $a(\cdot)$, $F$ in system \eqref{11} and moreover the system itself to $ C_{3R}$ so that only local estimate in Theorem~\ref{main theorem} is used to prove Theorem~\ref{main theorem2}. This section is divided into two steps. In the first step, we extend \eqref{11} to the spatial direction. In the second step, we extend the first step system to the time direction.

\subsection{Extension along the lateral boundary}
We observe that topological boundary of $ C_R=D_R\times I_R$ in the spatial direction consists of hyper planes $\{x\in \RR^{n+1}:  x_i =\pm R\}$ for $1\le i\le n$. Since the argument is analogous, we only consider when $x_n=-R$. We denote $x'=(x_1,...,x_{n-1})\in \RR^{n-1}$ for $x=(x_1,...,x_{n-1},x_n)\in \RR^n$. For each $(x',x_n,t)\in  C_R$, we define the reflection map $\mathcal{R}$ along the hyperplane $\{\RR^{n+1}: x_n=-R\}$ as $\mathcal{R}(x',x_n,t)=(x',-2R-x_n,t).$
It is easy to see $\mathcal{R}^{-1}\equiv \mathcal{R}$.
For $\varphi\in C_0^\infty( C_R\cup\mathcal{R}( C_R),\RR^N)$, we define $\phi$ to be $\phi(z)=\varphi(z)-\varphi\circ\mathcal{R}(z)$ for all $(x',x_n,t)\in  C_R$.
Since $\phi\equiv0$ on $\pa C_R$, it follows from the trace theorem that $\phi\in W_0^{1,\infty}( C_R,\RR^N)$ and $\phi$ is an admissible test function to \eqref{11}. We have
\begin{align}\label{53}
    \begin{split}
        &0=\iint_{ C_R}\left(-u\cdot\phi_t+b\mA(z,\na u)\cdot \na \phi-\mA(z,F)\cdot\na \phi\right)\,dz\\
         &\qquad=\iint_{ C_R}\left(-u\cdot\varphi_t+b\mA(z,\na u)\cdot\na\varphi-\mA(z,F)\cdot \na\varphi\right)\,dz\\
         &\qquad\qquad+\iint_{ C_R} \left(u\cdot (\varphi\circ\mathcal{R})_t-b\mA(z,\na u)\cdot \na (\varphi\circ\mathcal{R})+\mA(z,F)\cdot\na(\varphi\circ\mathcal{R})\right)\,dz.
    \end{split}
\end{align}
We will apply the change of variables to $x_n$ in order to replace the referenced domain $ C_R$ by $\mathcal{R}( C_R)$. Firstly, note that the determinant of the Jacobian matrix $\mathcal{J}$ of $\mathcal{R}$ is $-1$. There holds
\begin{align*}
    \begin{split}
        &\iint_{ C_R}u\cdot (\varphi\circ\mathcal{R})_t\,dz=\iint_{ C_R}u\cdot(\varphi_t\circ\mathcal{R})\,dz\\
        &\qquad=\iint_{\mathcal{R}( C_R)} (u\circ\mathcal{R}^{-1})\cdot\varphi_t|\det\mathcal{J}|\,dz
        =\iint_{\mathcal{R}( C_R)}-(-u\circ\mathcal{R}^{-1})\cdot\varphi_t\,dz.
    \end{split}
\end{align*}
Secondly, we calculate the $p$-Laplace operator term to estimate the term involving $\mathcal{A}(z,\na u)$.
\begin{align*}
    \begin{split}
       &\iint_{ C_R} -b|\na u|^{p-2}\na u\cdot \na(\varphi\circ \mathcal{R})\,dz\\
       &\qquad=\sum_{1\le i\le n-1}\iint_{ C_R}-b|\na u|^{p-2}\pa_iu\cdot(\pa_i\varphi   \circ \mathcal{R})\,dz\\
       &\qquad\qquad+\iint_{ C_R}b|\na u|^{p-2}\pa_nu\cdot(\pa_n\varphi\circ\mathcal{R})\,dz\\
       &\qquad=\iint_{\mathcal{R}( C_R)}b\circ \mathcal{R}^{-1}|\na (-u\circ \mathcal{R}^{-1})|^{p-2}\na (-u\circ \mathcal{R}^{-1})\cdot \na\varphi\,dz.
       \end{split}
\end{align*}
The same argument holds when we replace $b$ and $p$ with $ba$ and $q$. Therefore, we deduce
\[\iint_{ C_R}-b\mA(z,\na u)\cdot \na (\varphi\circ\mathcal{R})\,dz
    =\iint_{\mathcal{R}( C_R)}b_{\mathcal{R}}\mA_{\mathcal{R}}(z,\na u_{\mathcal{R}})\cdot \na \varphi\,dz,\]
where
\begin{align*}
\begin{split}
   & u_{\mathcal{R}}=-u\circ \mathcal{R}^{-1},\\
   & b_{\mathcal{R}}=b\circ\mathcal{R}^{-1},\\
   & a_{\mathcal{R}}=a\circ \mathcal{R}^{-1},\\
   & \mA_{\mathcal{R}}(z,\xi)=|\xi|^{p-2}\xi+a_{\mathcal{R}}|\xi|^{q-2}\xi.
\end{split}
\end{align*}
Similarly, we also have
\[\iint_{ C_R}\mA(z,F)\cdot\na(\varphi\circ\mathcal{R}^{-1})\,dz
        =\iint_{\mathcal{R}( C_R)}\mA_{\mathcal{R}}(z,F_{\mathcal{R}})\cdot\na \varphi\,dz,\]
where $F_{\mathcal{R}}=(F_1\circ\mathcal{R}^{-1},...,F_{n-1}\circ\mathcal{R}^{-1},-F_n\circ\mathcal{R}^{-1}).$
Therefore, \eqref{53} becomes
\begin{align*}
    \begin{split}
        &0=\iint_{ C_R}\left(-u\cdot\varphi_t+b\mA(z,\na u)\cdot\na\varphi-\mA(z,F)\cdot \na\varphi\right)\,dz\\
        &\qquad+\iint_{\mathcal{R}( C_R)}\left(-u_{\mathcal{R}}\cdot\varphi_t+b_{\mathcal{R}}\mA_{\mathcal{R}}(z,\na u_{\mathcal{R}})\cdot\na\varphi-\mA_{\mathcal{R}}(z,F_{\mathcal{R}})\cdot \na\varphi\right)\,dz.
    \end{split}
\end{align*}
Extending $u$, $b$, $a$ and $F$ to $u_{\mathcal{R}}$, $b_{\mathcal{R}}$, $a_{\mathcal{R}}$ and $F_{\mathcal{R}}$ in $\mathcal{R}( C_R)$ as above displays, $u$ is a weak solution to
\[\begin{cases}
    u_t-\dv(b\mA(z,\na u))=-\dv\mA(z,F)&\text{in}\quad C_R\cup\mathcal{R}( C_R),\\
    u=0&\text{on}\quad\pa_p( C_R\cup\mathcal{R}( C_R)).
\end{cases}\]
Since $u\equiv0$ on $\pa_p C_R$, it is easy to see
\begin{align*}
\begin{split}
    &u\in C(I_{R};L^2( C_R\cup\mathcal{R}( C_R),\RR^N))\cap L^1(I_R;W_{0}^{1,1}( C_R\cup\mathcal{R}( C_R),\RR^N))\quad\text{with}\\
    &\qquad\qquad\iint_{ C_R\cup \mathcal{R}( C_R)}H(z,|\na u|)\,dz<\infty.
\end{split}
\end{align*}
Also, note that extended $b$ and $F$ satisfy the ellipticity condition \eqref{12} in $ C_R\cup\mathcal{R}( C_R)$ and $H(z,|F|)\in L^1( C_R\cup\mathcal{R}( C_R))$. It also follows that $a\in C^{\alpha,\alpha/2}( C_R\cup\mathcal{R}( C_R))$ with $[a]_{\alpha; C_R\cup\mathcal{R}( C_R)}=[a]_{\alpha; C_R}.$ Thus \eqref{15} holds in $ C_R\cup\mathcal{R}( C_R)$. Indeed for each $z=(x',x_n,t)\in  C_R$ and $w=(y',y_n,s)\in \mathcal{R}( C_R)$, there holds
\begin{align}\label{reflex}
\begin{split}
    &|z-\mathcal{R}^{-1}(w)|
    =|(x',x_n,t)-(y',-y_n-2R,s)|=|\mathcal{R}(z)-w|\\
    &\qquad=|(x',x_n+R,t)-(y',-(y_n+R),s)|\\
    &\qquad\qquad\le |(x',x_n,t)-(y',y_n,s)|=|z-w|
\end{split}
\end{align}
since the reflection makes the distance between points in $n$-variable shorter.
We now verify that the VMO condition \eqref{18} of $b$ in $ C_R$ implies the local VMO condition \eqref{115} in $ C_R\cup\mathcal{R}( C_R)$
\begin{align}\label{512}
    \lim_{r\to0^+}\sup_{\tau\le r^2}\sup\limits_{\substack{B_{r}(x_0)\times I_\tau(t_0)\\\subset  C_R\cup\mathcal{R}( C_R)}}
    \fiint_{ B_r(x_0)\times I_\tau(t_0)}|b(x,t)-b_{ B_{r}(x_0)\times I_\tau(t_0)}|\,dx\,dt=0.
\end{align}
Since $b$ is extended by even reflection, we may assume $z_0=(x_0',x_{0,n},t_0)\in  C_R$. For each $z=(x',x_n,t)\in (B_r(x_0)\times I_\tau(t_0)) \cap\mathcal{R}( C_R)$, again \eqref{reflex} leads $|(x_0',x_{0,n},t_0)-(x',-2R-x_n,t)|\le |(x_0',x_{0,n},t_0)-(x',x_n,t)|$. Therefore we have $\mathcal{R}(z)=(x',-2R-x_n,t)\in (B_r(x_0)\times I_\tau(t_0))\cap  C_R$ and
\begin{align*}
    \begin{split}
        &\fiint_{B_{r}(x_0)\times I_{\tau}(t_0)} |b(z)-b_{B_{r}(x_0)\times I_\tau(t_0)}|\,dz\\
        &\qquad\le 2\fiint_{B_{r}(x_0)\times I_{\tau}(t_0)} |b(z)-b_{(B_{r}(x_0)\times I_\tau(t_0))\cap  C_R}|\,dz\\
        &\qquad\qquad\le 4\fiint_{( B_{r}(x_0)\times I_{\tau}(t_0))\cap  C_{R}} |b(z)-b_{(B_{r}(x_0)\times I_\tau(t_0))\cap  C_R}|\,dz.
    \end{split}
\end{align*}
The last term goes to $0$ as $r$ approaches $0$ from \eqref{18}. Hence \eqref{512} holds true.

Inductively repeating extension arguments from the previous steps, the extended map $u$ is the weak solution to
\begin{align}\label{514}
    \begin{cases}
        u_t-\dv(b\mA(z,\na u))=-\dv\mA(z,F)&\text{in}\quad D_{3R}\times I_R,\\
        u=0&\text{on}\quad\pa_p(D_{3R}\times I_R),
    \end{cases}
\end{align}
where \eqref{12}, \eqref{14}, \eqref{15} and \eqref{16} holds by replacing the reference domain $ C_R$ with $D_{3R}\times I_R$, and  and \eqref{115} holds in $D_{3R}\times I_R$ whenever the center point $z_0$ belongs to $  C_R$.

\subsection{Extension along the time direction} In this subsection, we extend \eqref{514} to $D_{3R}\times (-9R^2,9R^2)$. 

\subsubsection{Initial boundary} We extend $u$, $F$ to be zero while, extend $b$ and $a$ evenly below the initial boundary $D_{3R}\times \{t=-R^2\}$. Again it is easy to see that
\begin{align*}
\begin{split}
    &u\in C((-9R^2,R^2);L^2(D_{3R},\RR^N))\cap L^1((-9R^2,R^2);W_0^{1,1}(D_{3R},\RR^N))\quad\text{with}\\
    &\qquad\qquad\iint_{D_{3R}\times (-9R^2,R^2)}H(z,|\na u|)\,dz<\infty.
\end{split}    
\end{align*}
Also, $b$ satisfies \eqref{12} and $a\in C^{\alpha,\alpha/2}$ with $[a]_{\alpha;D_{3R}\times (-9R^2,R^2)}=[a]_{\alpha;D_{3R}\times I_R}$. The VMO condition of $b$ again is satisfied as well. We omit the detailed proof since the argument is repeated from the previous subsection.

We verify \eqref{514} is extended to $D_{3R}\times (-9R^2,R^2)$. Let $\vartheta>0$ and $\zeta_{\vartheta}\in W^{1,\infty}(\RR)$ be a Lipschitz function defined as
\[\zeta_{\vartheta}(t)=
    \begin{cases}
        0&\text{if}\quad t\in (-\infty,-R^2)\\
        \frac{1}{\vartheta}(t+R^2)&\text{if}\quad t\in [-R^2,-R^2+\vartheta],\\
        1&\text{if}\quad t\in (-R^2+\vartheta,\infty).
    \end{cases}\]
For $\varphi\in C_0^\infty(D_R\times (-9R^2,R^2),\RR^N)$ there holds
\begin{align*}
    \begin{split}
        &\iint_{D_{3R}\times (-9R^2,R^2)}-u\cdot\varphi_t\,dz=\lim_{\vartheta\to0^+}\iint_{D_{3R}\times (-R^2,R^2)}-u\cdot\varphi_t\zeta_\vartheta\,dz\\
        &\qquad=\lim_{\vartheta\to0^+}\iint_{D_{3R}\times (-R^2,R^2)}-u\cdot(\varphi\zeta_\vartheta)_t\,dz+\lim_{\vartheta\to0^+}\iint_{D_{3R}\times (-R^2,R^2)}u\cdot\varphi\pa_t\zeta_\vartheta\,dz.
    \end{split}
\end{align*}
We observe from \eqref{111} that
\begin{align*}
    \begin{split}
        &\lim_{\vartheta\to0^+}\biggl|\iint_{D_{3R}\times (-R^2,R^2)}u\cdot\varphi\pa_t\zeta_\vartheta\,dz\biggr|
        =\lim_{\vartheta\to0^+}\fint_{-R^2}^{-R^2+\vartheta}\int_{D_R}|u\cdot\varphi|\,dz\\
        &\qquad\le \lim_{\vartheta\to0^+}\left(\fint_{-R^2}^{-R^2+\vartheta}\int_{D_R}|u|^2\,dz\right)^\frac{1}{2}\left(\fint_{-R^2}^{-R^2+\vartheta}\int_{D_R}|\varphi|^2\,dz\right)^\frac{1}{2}\\
        &\qquad\qquad= 0\cdot\left(\int_{D_R}|\varphi(x,-R^2)|^2\,dz\right)^\frac{1}{2}=0.
    \end{split}
\end{align*}
Thus, \eqref{11} gives
\begin{align*}
    \begin{split}
        &\iint_{D_{3R}\times (-9R^2,R^2)}-u\cdot\varphi_t\,dz\\
        &\qquad=\lim_{\vartheta\to0^+}\iint_{D_{3R}\times (-R^2,R^2)}\left(-b(z)\mA(z,\na u)\cdot \na\varphi\zeta_\vartheta+\mA(z,F)\cdot \na\varphi\zeta_\vartheta\right)\,dz\\
        &\qquad\qquad=\iint_{D_{3R}\times (-R^2,R^2)}\left(-b(z)\mA(z,\na u)\cdot \na\varphi+\mA(z,F)\cdot \na\varphi\right)\,dz\\
        &\qquad\qquad\qquad=\iint_{D_{3R}\times (-9R^2,R^2)}\left(b(z)\mA(z,\na u)\cdot \na\varphi+\mA(z,F)\cdot \na\varphi\right)\,dz.
    \end{split}
\end{align*}
It follows that a trivial extension of $u$ is a weak solution to
\begin{align}\label{519}
\begin{cases}
    u_t-\dv (b(z)\mA(z,\na u))=-\dv\mA(z,F)&\text{in}\quad D_{3R}\times (-9R^2,R^2),\\
    u=0&\text{on}\quad \pa_p(D_{3R}\times (-9R^2,R^2)).
\end{cases}
\end{align}

\subsubsection{Topological boundary} In this case, we extend $b$ and $a$ evenly along $\{t=R^2\}$ in $D_{3R}\times (R^2,9R^2)$ whereas we extend $F$ to be zero on $D_{3R}\times (R^2,\infty)$.
Again the ellipticity condition and the VMO condition of $b$, H\"older's continuity of $a$ hold. Since we consider the case $\inf a>0$ in $D_{3R}\times (-9R^2,9R^2)$, there exists a weak solution $w$ to
\[\begin{cases}
        w_t-\dv(b(z)\mathcal{A}(z,\na w))=-\dv \mathcal{A}(z,F) &\text{in}\quad D_{3R}\times (-9R^2,9R^2),\\
        w=0&\text{in}\quad \pa_p(D_{3R}\times (-9R^2,9R^2)).
    \end{cases}\]
The above system is equivalent to \eqref{519} in $D_{3R}\times (-9R^2,R^2)$ with the same boundary data on $\pa_p(D_{3R}\times (-9R^2,R^2))$. The uniqueness theorem for the parabolic $q$-Laplace system says $w\equiv u$ in $D_{3R}\times (-9R^2,R^2)$ and $w$ is a extension of $u$. Therefore $u$ is a weak solution to
\begin{align}\label{523}
    u_t-\dv(b(z)\mA(z,\na u))=-\dv\mA(z,F)\quad \text{in}\quad C_{3R}.
\end{align}
We are ready to prove Theorem~\ref{main theorem2}. We are enough to consider when $\sigma\in(1+\varepsilon_0,\infty)$. Recall $\data$ depends on
\[n,N,p,q,\alpha,\nu,L,[a]_{\alpha},R,\|u\|_{L^{\infty}(I_{3R};L^2(D_{3R}))},
    \|H(z,|\na u|)\|_{L^1( C_{3R})},\|H(z,|F|)\|_{L^1( C_{3R})}.\]
On the other hand, we deduce from the energy estimate that
\[\sup_{t\in (-9R^2,9R^2)}\int_{D_{3R}}|u|^2\,dx+\iint_{ C_{3R}}H(z,|\na u|)\,dz\le \iint_{ C_{3R}}H(z,|F|)\,dz\]
and from the trivial extension that
\[\iint_{ C_{3R}}(H(z,|F|))^{\kappa}\,dz\le 3^{n}\iint_{ C_{R}}(H(z,|F|))^{\kappa}\,dz\]
for all $\kappa\in(1,\infty)$.
This leads $\data=\data_g$ and $\varepsilon_0=\varepsilon_0(\data_g)$. Using the estimate in Theorem~\ref{main theorem} to \eqref{523} in $ C_{3R}$ and covering argument and energy estimates, there exists  $c=c(\data_g,\|a\|_{\infty},\sigma,\|H(z,|F|)\|_{1+\varepsilon_0})$ such that
\[\fiint_{ C_{R}}(H(z,|\na u|))^\sigma
    \le c\left(\fiint_{ C_{R}}(H(z,|F|))^\sigma\,dz+1\right)^{\frac{q}{2}}.\]
This completes the proof.

\section{Proof of Corollary~\ref{main theorem3}}
We apply the uniqueness and existence result in \cite[Theroem 2.6 and Theorem 2.7]{KKS}. There exists a sequence of weak solutions $\{u_l\}_{l\in \mathbb{N}}\subset L^q(I_R;W_0^{1,q}(D_R,\RR^N))$ to the Dirichlet boundary problem
\[\begin{cases}
        \pa_tu_l-\dv (b(z)\mA_l(z,\na u_l))=-\dv \mA_l(z,F_l)&\text{in}\quad C_R\\
        u_l=0&\text{on}\quad\pa_p C_R,
    \end{cases}\]
 such that
\begin{align}\label{62}
    \lim_{l\to\infty}\iint_{ C_R}H(z,|\na u-\na u_l|)\,dz=0,
\end{align}
where $\{F_l\}_{l\in \mathbb{N}}\subset L^\infty(Q_R,\RR^{Nn})$ is a sequence of truncated functions of $F$ satisfying
\[\lim_{l\to\infty}\iint_{ C_R}H(z,|F-F_l|)\,dz=0\]
and $\mA_l$ is a perturbed $q$-Laplace operator with the positive decreasing sequence $\{\ep_l\}_{l\in\mathbb{N}}$ such that
\begin{align}\label{64}
    \mA_l(z,\xi)=|\xi|^{p-2}\xi+a_l(z)|\xi|^{q-2}\xi,\quad a_l(z)=a(z)+\ep_l,\quad \lim_{l\to\infty}\ep_l=0.
\end{align}
Furthermore, it follows from the proof of \cite[Theorem 2.5]{KKS} that for $H_l(z,s)=s^p+a_l(z)s^q$
\begin{align}\label{65}
    H_l(z,|F_l(z)|)\le 2H(z,|F(z)|)\quad\text{for all}\quad z\in  C_R.
\end{align}
We only consider when $\sigma\in(1+\varepsilon_0,\infty)$. For each $l$, the estimate in Theorem~\ref{main theorem2} gives
\[\fiint_{ C_{R}}(H_l(z,|\na u_l|))^\sigma\,dz\le c\left(\fiint_{ C_R}(H_l(z,|F_l|))^\sigma\,dz+1\right)^\frac{q}{2},\]
where $c=c(\data_g,R,\|a_l\|_{\infty},\sigma,\|H(z,|F_l|)\|_{1+\varepsilon_0})$. Now applying \eqref{64} and \eqref{65}, we get
\[\fiint_{ C_{R}}(H(z,|\na u_l|))^\sigma\,dz\le c\left(\fiint_{ C_R}(H(z,|F|))^\sigma\,dz+1\right)^\frac{q}{2}\]
with $c=c(\data_g,\|a\|_{\infty},\sigma,\|H(z,|F|)\|_{1+\varepsilon_0})$. This implies $H(z,|\pa_i u_l^j|)$ is uniformly bounded in $L^\sigma( C_R)$ for each $1\le i\le n$ and $1\le j\le N$. Thus, there exists $0\le v_i^j(z)\in L^\sigma( C_R)$ such that $H(z,|\pa_i u_l^j(z)|)$ weakly converges to $v_i^j$ in $L^\sigma( C_R)$. Since $H(z,s)$ is convex increasing function for each $z\in C_R$ with $H(z,0)=0$, we are able to find $0\le w_i^j(z)$ such that $v_i^j(z)=H(z,w_i^j(z))$.
Meanwhile, $H(z,|\pa_iu_l^j(z)|)$ converges point-wisely to $H(z,|\pa_i u^j(z)|)$ on account of \eqref{62}. Consequently, $H(z,w_i^j(z))\equiv H(z,|\pa_iu^j(z)|)$ holds from \cite[Chapter V, Proposition 9.1c]{MR1897317}. Hence $w_i^j\equiv \pa_iu^j$ holds and we obtain
\begin{align*}
\begin{split}
    &\fiint_{ C_{R}}(H(z,|\na u|))^\sigma\,dz
    \le \liminf_{l\to\infty}\fiint_{ C_{R}}(H(z,|\na u_l|))^\sigma\,dz\\
    &\qquad\le c\left(\fiint_{ C_R}(H(z,|F|))^\sigma\,dz+1\right)^\frac{q}{2}.
\end{split}
\end{align*}
This completes the proof.


\end{document}